\documentclass[12pt]{article}

\usepackage{fancyhdr}
\usepackage{setspace}
\usepackage{amsmath}
\usepackage{amsfonts}
\usepackage{hyperref}
\usepackage{amssymb}
\usepackage{mdwlist}
\usepackage{amscd}
\usepackage{lipsum}
\usepackage{mathrsfs}
\usepackage{mdframed}
\usepackage{amsthm}
\usepackage{enumitem}
\usepackage{verbatim}
\usepackage{graphicx}
\usepackage{epstopdf}
\usepackage[utf8]{inputenc}
\usepackage[english]{babel}
\usepackage{float}
\usepackage{stmaryrd}
\usepackage{tikz}
\usepackage{comment}
\usepackage{csquotes}
\usepackage{cleveref}
\usepackage[auth-sc]{authblk}

\setenumerate[0]{label=(\alph*)}

\theoremstyle{plain}

\crefname{claim}{claim}{claims}
\newtheorem{thm}{Theorem}
\crefname{thm}{theorem}{theorems}
\newtheorem{lem}[thm]{Lemma}
\crefname{lem}{lemma}{lemmas}
\newtheorem{prop}[thm]{Proposition}
\crefname{prop}{proposition}{propositions}

\crefname{exer}{exercise}{exercises}
\newtheorem*{cor}{Corollary}
\crefname{cor}{corollary}{corollaries}

\theoremstyle{definition}
\newtheorem*{defn}{Definition}
\crefname{defn}{definition}{definitions}
\newtheorem*{qn}{Question}
\newtheorem*{qns}{Questions}
\newtheorem{obs}[thm]{Observation}
\theoremstyle{remark}
\newtheorem{rmk}{Remark}[section]
\crefname{rmk}{remark}{remarks}

\def\Fraisse{Fra\"{i}ss\'{e}}

\def\Q{\mathbb{Q}}
\def\Z{\mathbb{Z}}

\def\O{\mathcal{O}}

\DeclareFontFamily{U}{min}{}
\DeclareFontShape{U}{min}{m}{n}{<-> udmj30}{}

\def\id{\operatorname{id}}
\def\from{:}

\newcommand{\paren}[1]{\left( #1 \right)}
\renewcommand{\brack}[1]{\left[ #1 \right]}
\renewcommand{\brace}[1]{\left\{ #1 \right\}}

\newcommand{\absval}[1]{\left| #1 \right|}

\newcommand{\Aut}[1]{\operatorname{Aut} \paren{#1}}

\newcommand{\dom}[1]{\operatorname{dom} \paren{ #1 }}
\newcommand{\ran}[1]{\operatorname{range} \paren{ #1 }}
\newcommand{\dur}[1]{\operatorname{dur} \paren{ #1 }}
\renewcommand{\sp}[1]{\operatorname{sp} \paren{ #1 }}
\newcommand{\parity}[1]{\operatorname{par} \paren{ #1 }}

\def\tp{\operatorname{tp}}

\def\defiff{\stackrel{\textrm{def}}{\iff}}

\newcommand{\orb}[2]{\O_{#1} \paren{#2}}
\newcommand{\OQ}[2]{\O_{#1} \brack{#2}}

\def\dhr{{\downharpoonright}}

\def\isom{\cong}

\hypersetup{
    colorlinks=true,
    linkcolor=red,
    urlcolor=blue
}

\title{The structure of generic automorphisms of the random poset}
\author{Dakota Thor Ihli\thanks{Department of Mathematics, University of Illinois at Urbana-Champaign, Urbana, Illinois. Email: \href{mailto:dihli2@illinois.edu}{\texttt{dihli2@illinois.edu}}.}}
\date{\vspace{-4ex}}

\usepackage[margin=2.75cm]{geometry}
\usepackage[labelfont=sc]{caption}

\numberwithin{thm}{section}

\begin{document}

\maketitle
\begin{abstract}
    We examine properties of generic automorphisms of the random poset, with the goal of explicitly characterizing them. We associate to each automorphism an auxiliary first-order structure, consisting of the random poset equipped with an infinite sequence of binary relations which encode the action of the automorphism. We then explicitly characterize generic automorphisms in terms of properties of this structure. Two notable such properties are ultrahomogeneity, and universality for a certain class of finite structures in this language. As this auxiliary structure seems to be new, we also address some model-theoretic questions. In particular, this structure fails to be saturated, and its theory neither is $\omega$-categorical nor admits quantifier-elimination, in contrast to many known ultrahomogeneous structures in finite languages.
    
    We also examine orbitals --- order-convex hulls of orbits --- and their use in describing automorphisms. In particular, we introduce and use new orders on the space of orbitals.
\end{abstract}
\tableofcontents

\section{Introduction}

A central question in the study of any Polish group is whether the group admits a comeagre conjugacy class, whose elements have come to be called \textit{generic elements}. The class of \textit{non-archimedean} groups is heavily studied, due to the connections with model theory. Indeed, the non-archimedean Polish groups precisely coincide\footnote{By definition, a topological group is non-archimedean if it admits a local basis at identity of clopen subgroups. The usual characterization is that a Polish group is non-archimedean precisely when it is topologically isomorphic to a closed subgroup of $S_{\infty}$.} with the automorphism groups of the highly homogeneous first-order structures known as \textit{\Fraisse{} limits}.

Many automorphism groups of these structures not only admit generic elements, but also those generic elements admit concrete descriptions, often in terms of the pointwise combinatorics of the action of the automorphism on the structure. In contrast, there are other groups of automorphisms which admit generics, but the method of proving this is more indirect. Indeed, Truss \cite{Tru92}, and later Kechris \& Rosendal \cite{KR04}, have identified amalgamation conditions on the set of finite partial automorphisms of the \Fraisse{} limit, which imply the existence of generics without constructing them.

Here, we are concerned with the case of the \textit{random poset} --- the unique (up to isomorphism) universal ultrahomogeneous countable poset, which we denote throughout by $\mathbf{P}$. Indeed, $\mathbf{P}$ is a \Fraisse{} limit; moreover, as the name suggests, it may also be constructed probabilistically, as described by Droste \& Kuske \cite{DK03}. Using the indirect methods alluded to above (and elaborated on later), Kuske \& Truss \cite{KT01} showed that $\Aut{\mathbf{P}}$ admits a comeagre conjugacy class. The main result of the present paper is a more concrete description of the generic elements, one which does not mention finite partial automorphisms.

We achieve this by associating to each $f \in \Aut{\mathbf{P}}$ an auxiliary structure $\mathbf{P}_{f}$, defined by equipping $\mathbf{P}$ with $\aleph_{0}$-many binary relations that encode the action of $f$ on $\mathbf{P}$. This structure captures the conjugacy relation precisely; that is, $f$ and $g$ are conjugate in $\Aut{\mathbf{P}}$ if and only if $\mathbf{P}_{f} \isom \mathbf{P}_{g}$. Our description of generic automorphisms is then given in terms of properties of the structure $\mathbf{P}_{f}$, which we now briefly describe.

To facilitate the study of automorphisms, we collect several tools for describing their actions on $\mathbf{P}$. For example, using the order on $\mathbf{P}$ allows one to quantify when, and how, a point is moved ``upward'' or ``downward'' by an automorphism. Indeed, we let the \textit{spiral length} of $x \in \mathbf{P}$ under $f \in \Aut{\mathbf{P}}$ be the least integer $n \geq 1$ for which $x$ and $f^{n} \paren{x}$ are comparable, and we say $x$ has infinite spiral length if no such $n$ exists. We then say $x$ has \textit{parity} $+1$ or $-1$ if $x < f^{n} \paren{x}$ or $x > f^{n} \paren{x}$, respectively. If either $f^{n} \paren{x} = x$, or else if $x$ has infinite spiral length, we say $x$ has parity $0$. The notions of parity and spiral length were introduced in \cite{KT01}; in \cref{sec:GeneralPosets}, we expand on their basic properties in the general framework of arbitrary posets.

Furthermore, for each $f \in \Aut{\mathbf{P}}$ and $i \in \Z$, we define the binary relation $b_{i}^{f}$ on $\mathbf{P}$ by letting $b_{i}^{f} \paren{x,y} \defiff x \leq f^{i} \paren{y}$. Then $\mathbf{P}_{f}$ is defined to be the first-order structure in the language $L := \brace{b_{i} : i \in \Z}$ obtained by equipping $\mathbf{P}$ with the $b_{i}^{f}$ relations. For convenient terminology, we often identify $b_{i}^{f} \paren{x,y}$ with its truth value in $\brace{0,1}$ and let $\mathbf{b}^{f} \paren{x,y} := \paren{b_{i}^{f} \paren{x,y}}_{i \in \Z}$ denote the resulting bi-infinite binary sequence. The $b_{i}$'s were also defined in \cite{KT01}, but they were not treated as a first-order language as they are here. As such, the structure $\mathbf{P}_{f}$ is new, and it is of independent interest.

Luckily, there are situations when we are able to determine the infinite behavior of $\mathbf{b}^{f} \paren{x,y}$ using only finite information. For example, suppose $x,y \in \mathbf{P}$ with $x < y$. If $f \in \Aut{\mathbf{P}}$, and $c \in \mathbf{P}$ is a fixed point of $f$ with $x < c < y$, then $x < c = f^{i} \paren{c} < f^{i} \paren{y}$ for every $i$, hence $b_{i}^{f} \paren{x,y} = 1$ for all $i \in \Z$. More examples of these finite configurations are collected in \cref{subsec:UsefulCombinatorialCriteria}. Conversely, one may want certain infinite patterns to be witnessed by finite configurations. We will show that this does, in fact, occur for generic automorphisms.

We are now able to sketch our characterization.
\begin{thm}\label{CharacterizationTheoremSketch}
    An automorphism $f \in \Aut{\mathbf{P}}$ is generic if and only if the following hold:
    \begin{enumerate}[label=(\Alph*)]
        \item $f$ has dense conjugacy class;
        \item The structure $\mathbf{P}_{f}$ is ultrahomogeneous (as an $L$-structure);
        \item For all $x \in \mathbf{P}$ with non-zero parity, the sequence $\mathbf{b}^{f} \paren{x,x}$ is eventually constant on both sides;
        \item\label{VagueDStatement} For all $x,y \in \mathbf{P}$, there exists a certain finite configuration, called tightening spirals, that forces $\mathbf{b}^{f} \paren{x,y}$ to be periodic on both sides.
    \end{enumerate}
\end{thm}

We refer to \cref{subsec:CharacterizingGenericAutomorphisms} for details, including the precise statement of \Cref{CharacterizationTheoremSketch} --- in particular, \Cref{GenericIsInSigma,SigmaIsGeneric} state sufficiency and necessity, respectively. We remark that the density of the conjugacy class of $f$ implies a certain amount of universality for $\mathbf{P}_{f}$; that is, many finite $L$-structures can be realized as substructures of $\mathbf{P}_{f}$.

Another tool we make use of are \textit{orbitals}, which were used in \cite{Tru92} as convenient terminology to describe generic automorphisms of the linear order $\Q$. The notion (which immediately generalizes to arbitrary posets) is even more valuable here. We define the orbital of a point $x$ under an automorphism $f$ to be the order-convex hull of the $f$-orbit of $x$; see \cref{subsec:Orbitals} for details. Once $f$ is fixed, orbitals partition the poset, and we call the resulting quotient the \textit{orbital quotient}. The parity of a point is orbital-invariant, so we may define the parity of an orbital to be the parity of its points. Furthermore, an orbital of parity $0$ consists of a single orbit, so spiral length is also orbital-invariant in the parity $0$ case. Orbitals of non-zero parity do not enjoy this property, however; for generic $f \in \Aut{\mathbf{P}}$, this fails in the most extreme way possible.
\begin{thm}\label{AllSpiralLengthsSketch}
    For generic $f \in \Aut{\mathbf{P}}$, all orbitals of non-zero parity admit points of all finite spiral lengths.
\end{thm}
This is stated more precisely as \Cref{NonZeroParityHasAllSpiralLengths}.

The orbital quotient inherits order structure from the original poset. Indeed, Truss's characterization of the generic automorphisms of $\Q$ was in terms of this order's properties. To generalize to arbitrary partial orders, we introduce two novel notions of order on the orbital quotient, called the \textit{strong} and \textit{weak} orders. As the name suggests, orbitals which are related in the strong order are related in the weak order as well; moreover, both generalize the notion used by Truss, as they coincide when the underlying poset is in fact a linear order. For details, we refer to \cref{subsec:OrdersOnOrbitals}.

The orders on orbital quotients allow a natural analogue of Truss's result about $\Q$ to be formulated for $\mathbf{P}$, and we are able to prove one direction of this analogue.
\begin{thm}\label{QuotientIsRandomSketch}
    Let $f \in \Aut{\mathbf{P}}$ be generic. The orbital quotient of $\mathbf{P}$ induced by $f$, equipped with the strong order, is isomorphic to the random poset. Furthermore, for any possible parity (and for parity $0$, spiral length), the orbitals of that parity (and for parity $0$, that spiral length) are dense in the orbital quotient.
\end{thm}
\begin{qn}
    Does the converse hold? In other words, does this condition on the orbital quotient imply genericity?
\end{qn}
For details, we refer to \cref{subsec:StructureOfTheOrbitalQuotient}; this theorem is stated more precisely as \Cref{GenericOrbitalQuotientIsRandom}.

We conclude with remarks on other model-theoretic properties of $\mathbf{P}_{f}$. As mentioned above, the $\mathbf{P}_{f}$ is ultrahomogeneous for the generic $f \in \Aut{\mathbf{P}}$. However, it fails to exhibit other nice model-theoretic properties, namely saturation, quantifier-elimination, and $\omega$-categoricity. This contrasts with $\mathbf{P}$ itself, which does enjoy these properties (as do many other \Fraisse{} structures in finite languages). Each of these negative results are essentially due to \Cref{AllSpiralLengthsSketch}. For details, we refer to \cref{subsec:ModelTheoreticConsiderations}.

\subsubsection*{Acknowledgements}
We extend gratitude to:
\begin{itemize}
    \item Aleksandra Kwiatkowska, for originally communicating the motivating problem to the author, and for providing useful references;
    \item The author's advisor Anush Tserunyan, for invaluable guidance and critique;
    \item Ronnie Chen, for helpful comments and suggestions for further research;
    \item Elliot Kaplan and Alexi Block Gorman, for helpful discussion of potential model-theoretic properties of the $\mathbf{P}_{f}$'s.
\end{itemize}

\section{General posets}\label{sec:GeneralPosets}
Throughout this section, $\paren{P,\leq}$ will denote an arbitrary poset.

\subsection{Preliminary notions}
We let $\Aut{P}$ denote the automorphism group of $P$. We say $x,y \in P$ are \textbf{comparable} if either $x \leq y$ or $y \leq x$. We write $x \parallel y$ if $x$ and $y$ are comparable, and we write $x \perp y$ if $x$ and $y$ are not comparable.

For a partial map $p \from P \rightharpoonup P$, we write $\dom{p}$ and $\ran{p}$ for the domain and range respectively, and we let $\dur{p}$ abbreviate $\dom{p} \cup \ran{p}$. (``dur'' is short for ``domain union range''.)

We say a partial map $p \from P \rightharpoonup P$ is a \textbf{partial automorphism} of $P$ if $x \leq y \iff p \paren{x} \leq p \paren{y}$ for all $x,y \in \dom{p}$. (By antisymmetry, all such maps are necessarily injective.) We then let $\mathcal{P}$ denote the set of those partial automorphisms $p$ with finite domain. Given $p \in \mathcal{P}$, we let $\brack{p} := \brace{f \in \Aut{P} : p \subseteq f}$. (A priori, $\brack{p}$ may be empty.)

For a set $A \subseteq P$ and a map $f$ on $P$, we denote $f^{\Z} \brack{A} := \bigcup_{n} f^{n} \brack{A}$. Furthermore if $x \in P$, we write $f^{\Z} \paren{x} := f^{\Z} \brack{\brace{x}}$, and call this set the \textbf{$f$-orbit} of $x$.

The following will be useful to state explicitly, especially since it's true for infinite posets too:

\begin{lem}\label{StrongAmalgamationProperty}
    Posets have the following strong form of the amalgamation property: for every poset $A$, and all posets $P_{1}$ and $P_{2}$ for which there exist order embeddings $f_{i} \from A \hookrightarrow P_{i}$ (i=1,2), there exists a poset $Q$ and order embeddings $g_{i} \from P_{i} \hookrightarrow Q$ (i=1,2) such that:
    \begin{enumerate}[label=(\alph*)]
        \item $g_{1} \circ f_{1} = g_{2} \circ f_{2}$;
        \item $Q = \ran{g_{1}} \cup \ran{g_{2}}$;
        \item $\ran{g_{1} \circ f_{1}} = \ran{g_{2} \circ f_{2}} = \ran{g_{1}} \cap \ran{g_{2}}$;
        \item $g_{1} \paren{x} \leq g_{2} \paren{y}$ if, and only if, there exists $a \in A$ such that $x \leq f_{1} \paren{a}$ and $f_{2} \paren{a} \leq y$, and similarly for $g_{2} \paren{y} \leq g_{1} \paren{x}$.
    \end{enumerate}
\end{lem}
\begin{proof}
    Define $Q$ on the disjoint union $P_{1} \sqcup \paren{P_{2} \setminus \ran{f_{2}}}$, letting $g_{1} := \id_{P_{1}}$ and $$g_{2} \paren{y} := \begin{cases}
    y & \text{if $y \notin \ran{f_{2}}$;}\\
    f_{1} \paren{f_{2}^{-1} \paren{y}} & \text{if $y \in \ran{f_{2}}$.}
    \end{cases}$$
    This gives (a), (b), and (c). Define the order on $Q$ by letting: $$x <_{Q} x' \defiff \begin{cases}
    x <_{P_{1}} x' & \text{if $x,x' \in P_{1}$;}\\
    x <_{P_{2}} x' & \text{if $x,x' \in P_{2}$;}\\
    \exists a \in A \paren{x <_{P_{1}} f_{1} \paren{a} \wedge f_{2} \paren{a} <_{P_{2}} x'} & \text{if $x \in P_{1}, x' \in P_{2}$;}\\
    \exists a \in A \paren{x <_{P_{2}} f_{2} \paren{a} \wedge f_{1} \paren{a} <_{P_{1}} x'} & \text{if $x' \in P_{1}, x \in P_{2}$.}
    \end{cases}$$
    Then one checks that $<_{Q}$ is a partial order on $Q$ making $g_{1}$ and $g_{2}$ order embeddings, and satisfying (d).
\end{proof}

We will often use back-and-forth methods for constructing isomorphisms. We will often use the following observation, which can be seen as a consequence of \Cref{StrongAmalgamationProperty}: 
    
\begin{obs}[Sandwich principle]\label{SandwichPrinciple}
    To add a new point $z$ to a finite poset $P$, it is necessary and sufficient to specify the (possibly empty) set $A$ of desired immediate predecessors of the new point, and the (possibly empty) set $B$ of desired immediate successors of the new point, subject only to the constraint that $x < y$ for all $x \in A$ and $y \in B$.
\end{obs}

We are unaware of an already-existing name for this notion. We have chosen the name to be illustrative: the immediate predecessors and successors form the bread\footnote{In case either of the sets of predecessors/successors is empty, we are allowing open-faced sandwiches.}, and the new point in the middle is the filling. We will slightly expand on this analogy later.

\subsection{Orbitals}\label{subsec:Orbitals}

\begin{defn}
    Let $f \in \Aut{P}$. We define the binary relation $\sim_{f}$ on $P$ by: $$x \sim_{f} y \defiff \textnormal{$\exists i,j \in \Z$ such that $f^{i} \paren{x} \leq y \leq f^{j} \paren{x}$}.$$
\end{defn}

\begin{prop}\label{OrbitalsPartitionPoset}
    For $f \in \Aut{P}$, $\sim_{f}$ is an equivalence relation on $P$.
\end{prop}
\begin{proof}
    Taking $i = j = 0$, clearly $x \sim_{f} x$.

    Suppose $f^{i} \paren{x} \leq y \leq f^{j} \paren{x}$. Then $f^{-j} \paren{y} \leq x \leq f^{-i} \paren{y}$, and so $x \sim_{f} y$ implies $y \sim_{f} x$. We will remark here that the same proof implies $x \sim_{f} y \iff x \sim_{f^{-1}} y$.

    Suppose $f^{i_{1}} \paren{x} \leq y \leq f^{j_{1}} \paren{x}$ and $f^{i_{2}} \paren{y} \leq z \leq f^{j_{2}} \paren{y}$. Then $f^{i_{1} + i_{2}} \paren{x} = f^{i_{2}} \paren{f^{i_{1}} \paren{x}} \leq f^{i_{2}} \paren{y} \leq z \leq f^{j_{2}} \paren{y} \leq f^{j_{2}} \paren{f^{j_{1}} \paren{x}} = f^{j_{1} + j_{2}} \paren{x}$. Hence $x \sim_{f} y$ and $y \sim_{f} z$ imply $x \sim_{f} z$.
\end{proof}

\begin{defn}
    For $f \in \Aut{P}$, we let $\mathcal{O}_{f} \from P \to P / \sim_{f}$ denote the associated quotient map for $\sim_{f}$. We abbreviate the quotient $P / \sim_{f}$ as simply $\OQ{f}{P}$, called the \textbf{orbital quotient}. For $x \in P$, we call the equivalence class $\orb{f}{x}$ the \textbf{orbital} of $x$ under $f$.
\end{defn}

Observe that orbitals are orbit-invariant --- that is, $f^{\Z} \paren{x} \subseteq \orb{f}{x}$ for any $x \in P$. Orbitals are also \textbf{convex} --- that is, for all $y,z \in \orb{f}{x}$ and all $w \in P$, if $y < w < z$ then $w \in \orb{f}{x}$. Both of these facts are immediate from the definitions. Moreover, $\orb{f}{x}$ is the smallest convex subset of $P$ containing $f^{\Z} \paren{x}$; so, we may think of $\orb{f}{x}$ as the ``convex hull'' of $f^{\Z} \paren{x}$, where convexity is order-theoretic rather than geometric.

\begin{lem}\label{InterlacingOrbitalsEqual}
    Let $f \in \Aut{P}$ and $x,y \in P$. Then $x \sim_{f} y$ if and only if there exist $x', x'' \in \orb{f}{x}$ and $y', y'' \in \orb{f}{y}$ such that $x' \leq y'$ and $y'' \leq x''$.
\end{lem}
\begin{proof}
    The forward implication is immediate: if $f^{i} \paren{x} \leq y \leq f^{j} \paren{x}$, then taking $x' := f^{i} \paren{x}$, $x'' := f^{j} \paren{x}$, and $y' = y'' := y$ works.
    
    For the converse, take $f^{i_{1}} \paren{x'} \leq x'' \leq f^{j_{1}} \paren{x'}$ and $f^{i_{2}} \paren{y'} \leq y'' \leq f^{j_{2}} \paren{y'}$. Then $f^{i_{1}} \paren{x'} \leq f^{i_{1}} \paren{y'} = f^{i_{1} - i_{2}} \paren{f^{i_{2}} \paren{y'}} \leq f^{i_{1} - i_{2}} \paren{y''} \leq f^{i_{1} - i_{2}} \paren{x''} \leq f^{i_{1} - i_{2}} \paren{f^{j_{1}} \paren{x'}} = f^{i_{1} + j_{1} - i_{2}} \paren{x'}$, so by orbit-invariance, $y'' \sim_{f} f^{i_{1} - i_{2}} \paren{y''} \sim_{f} x'$. Hence, $x \sim_{f} x' \sim_{f} y'' \sim_{f} y$.
\end{proof}

\begin{defn}
    Given $f \in \Aut{P}$ and $x \in P$, we let the \textbf{spiral length} be given by $\sp{x,f} := \min \paren{\brace{\infty} \cup \brace{n \geq 1 : x \parallel f^{n} \paren{x}}}$. Define the \textbf{parity} of $f$ at $x$ by:
    $$\parity{x,f} := \begin{cases}
    +1 & \text{if $\sp{x,f} < \infty$ and $x < f^{\sp{x,f}} \paren{x}$;}\\
    -1 & \text{if $\sp{x,f} < \infty$ and $x > f^{\sp{x,f}} \paren{x}$;}\\
    0 & \text{otherwise.}
    \end{cases}$$
\end{defn}

\begin{lem}\label{SpiralLengthOrbitInvariant}
    Spiral length and parity are orbit-invariant; that is, $\sp{f^{k} \paren{x},f} = \sp{x,f}$ and $\parity{f^{k} \paren{x},f} = \parity{x,f}$ for any $k \in \Z$.
\end{lem}
\begin{proof}
    If $n \in \Z$, then since $f$ is an automorphism, $x < f^{n} \paren{x}$ if and only if $f^{k} \paren{x} < f^{k} \paren{f^{n} \paren{x}} = f^{n+k} \paren{x} = f^{n} \paren{f^{k} \paren{x}}$. Therefore $\sp{x,f} = \sp{f^{k} \paren{x},f}$ (since they're both the minimum of the same set of numbers), and moreover $\parity{f^{k} \paren{x},f} = \parity{x,f}$.
\end{proof}

\begin{lem}\label{ParityOrbitalInvariant}
    Parity is orbital-invariant; that is, $x \sim_{f} y$ implies $\parity{x,f} = \parity{y,f}$. Moreover, if this common parity is $0$, then $\sp{x,f} = \sp{y,f}$.
\end{lem}
\begin{proof}
    Suppose $f^{i} \paren{x} \leq y \leq f^{j} \paren{x}$. Then $x \leq f^{j-i} \paren{x}$, and $y \leq f^{j} \paren{x} = f^{j-i} \paren{f^{i} \paren{x}} \leq f^{j-i} \paren{y}$. Thus $\parity{x,f} = \parity{y,f} = \operatorname{sgn} \paren{j-i}$.
    
    If $\parity{x,f} = 0$, then $f^{i} \paren{x} = f^{j} \paren{x} = y$, and so $\orb{f}{x} = f^{\Z} \paren{x}$. Thus $\sp{x,f} = \sp{y,f}$ and $\parity{x,f} = \parity{y,f}$ by \Cref{SpiralLengthOrbitInvariant}.
\end{proof}

\begin{rmk}
    As we will see later, spiral length need not be orbital-invariant for orbitals of non-zero parity.
\end{rmk}

\subsection{\texorpdfstring{The language $L$}{The language L}}\label{subsec:TheLanguageL}

We begin with some motivating observations. For automorphisms $f,g \in \Aut{P}$, let $P'$ and $P''$ be the structures obtained by adding a unary function symbol whose interpretation in $P'$ is $f$ and whose interpretation in $P''$ is $g$. Then $h \from P \to P$ is an isomorphism from $P'$ to $P''$ if and only if $h \in \Aut{P}$ such that $g = hfh^{-1}$. Thus, the conjugacy relation on $\Aut{P}$ is precisely the isomorphism relation for these structures expanded by an automorphism.

However, this perspective does not play particularly well with finite partial automorphisms, which we show by example. Define $f,g \in \Aut{\Q}$ by $f \paren{x} = x+1$ and $g \paren{x} = x-1$, and let $\gamma \from \Q \rightharpoonup \Q$ be the partial map given by $\gamma = \brace{\paren{0,0}}$. Then $\gamma$ is vacuously a partial isomorphism\footnote{Recall that if a language has a function symbol, partial isomorphisms are considered with respect to the graph relation. Thus, $\gamma$ preserves the function symbol precisely when $f \paren{x} = y \iff g \paren{\gamma \paren{x}} = \gamma \paren{y}$ for all $x,y \in \dom{\gamma}$.} from $\paren{\Q,f}$ to $\paren{\Q,g}$. However, clearly $f$ and $g$ are not conjugate (since all points have positive parity for $f$ and negative for $g$), so $\gamma$ cannot extend to a full automorphism from $\paren{\Q,f}$ to $\paren{\Q,g}$. One may construct more complex examples to show this continues to fail even if we require $\gamma$ to preserve parities.

From this example, the reader may get the impression that extending finite partial isomorphisms to full isomorphisms would be far too ambitious of an expectation. However, by encoding the action of automorphisms with a different type of language, this becomes an attainable goal.

\begin{defn}
    We let $L$ be the language $\brace{b_{i} : i \in \Z}$, where each $b_{i}$ is a binary relation symbol. Then to each $f \in \Aut{P}$, we associate an $L$-structure -- denoted $P_{f}$ -- with the interpretation $b_{i}^{f} \paren{x,y} \defiff x \leq f^{i} \paren{y}$ for each $i \in \Z$. 
\end{defn}

We remark that $b_{i}^{f}$ was defined in \cite{KT01}, though it was not treated as a first-order language. For convenience of terminology, we often identify $b_{i}^{f} \paren{x,y}$ with its truth value in $2 = \brace{0,1}$, and we let $\mathbf{b}^{f} \paren{x,y}$ abbreviate the resulting sequence $\paren{b_{i}^{f} \paren{x,y}}_{i \in \Z} \in 2^{\Z}$. This allows us to use phrases like ``periodic'', ``eventually periodic'', and so on.

\begin{lem}\label{BDefinableRelations}
    Let $f \in \Aut{P}$.
    \begin{enumerate}
        \item The function $f$ and the order $\leq$ are $0$-definable in $P_{f}$.
        \item For each $1 \leq n < \infty$, the relation $\sp{x,f} = n$ is $0$-definable in $P_{f}$.\footnote{A priori, the same can not be said for the case $\sp{x,f} = \infty$.}
        \item For each $x,y \in P$ with $\sigma := \parity{x,f} \neq 0$ and $n := \sp{x,f}$, we have $b_{i}^{f} \paren{x,y} \implies b_{i + \sigma \cdot n}^{f} \paren{x,y}$ and $b_{i}^{f} \paren{y,x} \implies b_{i + \sigma \cdot n}^{f} \paren{y,x}$ for all $i \in \Z$. Thus, the sequences $\mathbf{b}^{f} \paren{x,y}$ and $\mathbf{b}^{f} \paren{y,x}$ are eventually periodic on both sides with period $n$.\footnote{The same need not hold for the case $\sp{x,f} = \infty$.}
        \item For each $x,y,z \in P$, we have $b_{i}^{f} \paren{x,y} \wedge b_{j}^{f} \paren{y,z} \implies b_{i+j}^{f} \paren{x,z}$.
    \end{enumerate}
\end{lem}

\begin{proof}
    Straightforward application of definitions. However, for (a) we will explicitly point out that $f^{k} \paren{x} = y$ if and only if $b_{-k}^{f} \paren{x,y}$ and $b_{k}^{f} \paren{y,x}$ hold.
\end{proof}

We claim that this language encodes conjugacy in $\Aut{P}$, in the same way adding a unary function symbol does.

\begin{prop}\label{ConjugateIffLIsomorphic}
    Let $f,g \in \Aut{P}$, let $A \subseteq P$ be $f$-invariant, and let $B \subseteq P$ be $g$-invariant. Let $h \from A \to B$ be a function. Then the following are equivalent:
    \begin{enumerate}[label=(\roman*)]
        \item $h$ is an order embedding with $h \circ f = g \circ h$;
        \item $h$ is an embedding of $L$-structures (where $A$ is a substructure of $P_{f}$ and $B$ is a substructure of $P_{g}$);
        \item $h$ is an embedding of $L'$-structures (where $L' := \brace{b_{i} : i=-1,0,1} \subseteq L$, and $A$ and $B$ are substructures in the associated reducts).
    \end{enumerate}
    In particular, $f$ and $g$ are conjugate in $\Aut{P}$ if and only if $P_{f}$ is isomorphic to $P_{g}$, if and only if their $L'$-reducts are isomorphic.
\end{prop}
\begin{proof}
    The ``in particular'' statement follows immediately from the equivalences, by letting $A = B = P$ and supposing $h$ is a bijection.
    \begin{itemize}
        \item (i) implies (ii):\\
        Since $h \circ f = g \circ h$, a two-sided induction shows $h \circ f^{k} = g^{k} \circ h$ for every $k \in \Z$. Then since $h$ is order-preserving, $b_{k}^{f} \paren{x,y} \iff x \leq f^{k} \paren{y} \iff h \paren{x} \leq h \paren{f^{k} \paren{y}} = g^{k} \paren{h \paren{y}} \iff b_{k}^{g} \paren{h \paren{x},h \paren{y}}$. Thus $h$ is an $L$-embedding.
        \item (ii) implies (iii):\\
        Trivial.
        \item (iii) implies (i):\\
        For each $x,y \in A$, we have $x \leq y \iff b_{0}^{f} \paren{x,y} \iff b_{0}^{g} \paren{h \paren{x},h \paren{y}} \iff h \paren{x} \leq h \paren{y}$. Hence $h$ is order-preserving.
        
        Moreover, for each $x,y \in A$: 
        \begin{align*}
            f \paren{x} = y &\iff b_{1}^{f} \paren{y,x} \wedge b_{-1}^{f} \paren{x,y}\\
            &\iff b_{1}^{g} \paren{h \paren{y},h \paren{x}} \wedge b_{-1}^{g} \paren{h \paren{x},h \paren{y}}\\
            &\iff g \paren{h \paren{x}} = h \paren{y} = h \paren{f \paren{x}}
        \end{align*}
        Thus $g \circ h = h \circ f$ as desired. \qedhere
    \end{itemize}
\end{proof}

In the example of $\Q$ at the beginning of the section, the defined $\gamma$ is not a partial isomorphism of the $L$-structures $P_{f}$ and $P_{g}$, since $b_{1}^{f} \paren{0, 0} \iff 0 \leq f \paren{0} = 1$ holds and $b_{1}^{g} \paren{\gamma \paren{0},\gamma \paren{0}} \iff 0 = \gamma \paren{0} \leq g \paren{\gamma \paren{0}} = g \paren{0} = -1$ fails.

Informally, since the language $L$ is infinite, a finite substructure of $P_{f}$ can still ``remember'' an infinite amount of information --- namely, the relations between entire orbits of points. Indeed, the following lemma makes this intuition precise.
\begin{lem}\label{FinitePartialLIsomCanExtend}
    Let $f,g \in \Aut{P}$. If $\gamma \from P_{f} \rightharpoonup P_{g}$ is a finite partial $L$-isomorphism, then $\gamma$ uniquely extends to a partial $L$-isomorphism $h \from f^{\Z} \brack{\dom{\gamma}} \to g^{\Z} \brack{\ran{\gamma}}$.
\end{lem}
\begin{proof}
    If such $h$ exists, it must be defined by letting $h \paren{f^{i} \paren{x}} := g^{i} \paren{\gamma \paren{x}}$ for each $x \in \dom{\gamma}$. The proof that this is well defined proceeds similarly to the proof of \Cref{ConjugateIffLIsomorphic}: for $x,y \in \dom{\gamma}$, we have $f^{i} \paren{x} \leq f^{j} \paren{y} \iff x \leq f^{j-i} \paren{y} \iff b_{j-i}^{f} \paren{x,y} \iff b_{j-i}^{g} \paren{\gamma \paren{x},\gamma \paren{y}} \iff \gamma \paren{x} \leq g^{j-i} \paren{\gamma \paren{y}} \iff g^{i} \paren{\gamma \paren{x}} \leq g^{j} \paren{\gamma \paren{y}}$, and similarly for $\geq$. Thus $f^{i} \paren{x} = f^{j} \paren{y} \iff g^{i} \paren{\gamma \paren{x}} = g^{j} \paren{\gamma \paren{y}}$ implies the map $h$ is a well-defined injection. Moreover, clearly $\ran{h} = g^{\Z} \brack{\ran{\gamma}}$, and $h \circ f = g \circ h$.
\end{proof}

\begin{defn}
    For $p \in \mathcal{P}$ and $x,y \in \dur{p}$, say $f, g \in \brack{p}$ are \textbf{isomorphic over $\paren{x,y}$} if $\mathbf{b}^{f} \paren{x,y} = \mathbf{b}^{g} \paren{x,y}$. We say $f,g$ are \textbf{isomorphic over $p$} if they are isomorphic over $\paren{x,y}$ for every $x,y \in \dur{p}$.
    
    We say $p$ \textbf{determines $\paren{x,y}$} if every two extensions $f,g \in \brack{p}$ are isomorphic over $\paren{x,y}$, and we say $p$ is \textbf{determined} if every two extensions $f,g \in \brack{p}$ are isomorphic over $p$.
\end{defn}

\begin{rmk}
    Isomorphism over $p$ was originally defined in \cite{KT01} as follows: $f$ and $g$ are isomorphic over $p$ if there is an order-isomorphism $\theta \from f^{\Z} \brack{\dom{p}} \to g^{\Z} \brack{\dom{p}}$ which fixes $\dom{p}$ and satisfies $\theta \circ f = g \circ \theta$. It is easily shown that this definition is equivalent to the one we give above, in light of \Cref{FinitePartialLIsomCanExtend} and \Cref{ConjugateIffLIsomorphic}.
\end{rmk}

Isomorphism over a pair of points is really a statement about their respective $p$-orbits, as the following lemma makes precise.

\begin{lem}\label{IsoOverPairOrbitInvariant}
    Let $p \in \mathcal{P}$ and $x,y \in \dur{p}$. Then $f,g \in \brack{p}$ are isomorphic over $\paren{x,y}$ iff they are isomorphic over $\paren{p^{k} \paren{x}, p^{\ell} \paren{y}}$\footnote{Whenever we use this notation, we are assuming $p$ is defined on all the intermediate steps, e.g. $0 \leq i < k$ implies $p^{i} \paren{x} \in \dom{p}$.} for any $k,\ell \in \Z$.
\end{lem}
\begin{proof}
    Suppose $f,g$ isomorphic over $\paren{x,y}$. For each $k,\ell \in \Z$, since $f$ extends $p$:
    \begin{align*}
        b_{i}^{f} \paren{p^{k} \paren{x}, p^{\ell} \paren{y}} &\iff p^{k} \paren{x} \leq f^{i} \paren{p^{\ell} \paren{y}}\\
        &\iff x \leq f^{i+\ell-k} \paren{y}\\
        &\iff b_{i+\ell-k}^{f} \paren{x, y}
    \end{align*}
    The same argument holds for $g$ also; the result follows.
\end{proof}

Moreover, suppose $p \in \mathcal{P}$ determines $\paren{x,y}$, and $q \supseteq p$. Then $q$ also determines $\paren{x,y}$. If $p$ is determined, and every $q$-orbit contains some $p$-orbit, then $q$ is also determined; in this case, we say $q$ is an \textbf{economical} extension of $p$, following the terminology of \cite{KT01}.

\subsection{Orders on orbitals}\label{subsec:OrdersOnOrbitals}

If $P$ is a linear order, this induces a natural linear order on $\OQ{f}{P}$ --- indeed, this linear order was used by Truss (\cite{Tru92}) to characterize the generic automorphism of $\Q$. For arbitrary posets, we can still recover order structure on $\OQ{f}{P}$ from $P$, but the ``correct'' definition is not immediately clear. We have identified two similar but distinct notions, both of which will come into play. Apart from the aforementioned paper of Truss, we are unaware of any study of orbitals and orders between them, so the following appears to be new.

\begin{defn}
    Let $f \in \Aut{P}$ and $x,y \in P$. We define the relations $<_{f}^{s}$ and $\leq_{f}^{w}$ --- called the \textbf{strong} and \textbf{weak} orders, respectively --- by: $$\orb{f}{x} <_{f}^{s} \orb{f}{y} \defiff \forall x' \sim_{f} x \, \forall y' \sim_{f} y \paren{x' < y'};$$
    $$\orb{f}{x} \leq_{f}^{w} \orb{f}{y} \defiff \exists x' \sim_{f} x \, \exists y' \sim_{f} y \paren{x' \leq y'}.$$
\end{defn}

\begin{prop}\label{OrbitalOrder}
    $<_{f}^{s}$ and $\leq_{f}^{w}$ are partial orders on $\OQ{f}{P}$.
\end{prop}
\begin{proof}
    Transitivity of both relations is a straightforward consequence of transitivity of the order on $P$. Since $x \nless x$, we have $\orb{f}{x} \nless_{f}^{s} \orb{f}{x}$, so $<_{f}^{s}$ is irreflexive. Thus $<_{f}^{s}$ is a (strict) partial order. Similarly, since $x \leq x$, we have $\orb{f}{x} \leq_{f}^{w} \orb{f}{x}$, so $\leq_{f}^{w}$ is reflexive. Moreover, antisymmetry is exactly \Cref{InterlacingOrbitalsEqual}. Thus $\leq_{f}^{w}$ is a (non-strict) partial order.
\end{proof}



\begin{prop}\label{OrbitalOrderWeakerCondition}
    Let $f \in \Aut{P}$ and $x,y \in P$.
    \begin{enumerate}[label=(\alph*)]
        \item The following are equivalent:
        \begin{enumerate}[label=(\roman*)]
            \item $\orb{f}{x} <_{f}^{s} \orb{f}{y}$;
            \item $f^{n} \paren{x} < y$ for every $n \in \Z$;
            \item $x < f^{n} \paren{y}$ for every $n \in \Z$;
            \item $\mathbf{b}^{f} \paren{x,y} = \mathbf{1}$ and $x \neq y$.
        \end{enumerate}
        \item The following are equivalent:
        \begin{enumerate}[label=(\roman*)]
            \item $\orb{f}{x} \leq_{f}^{w} \orb{f}{y}$;
            \item $f^{n} \paren{x} \leq y$ for some $n \in \Z$;
            \item $x \leq f^{n} \paren{y}$ for some $n \in \Z$;
            \item $\mathbf{b}^{f} \paren{x,y} \neq \mathbf{0}$.
        \end{enumerate}
    \end{enumerate}
\end{prop}
\begin{proof}
    \begin{itemize}
        \item (a)(i) implies (a)(ii), and (b)(ii) implies (b)(i):\\
        Immediate from the definitions and orbit-invariance of orbitals.
        \item (a)(ii) implies (a)(i):\\
        Suppose $x' \sim_{f} x$ and $y' \sim_{f} y$; by definition of $\sim_{f}$, there are $i, j \in \Z$ such that $x' \leq f^{i} \paren{x}$, and $f^{j} \paren{y} \leq y'$. By hypothesis, $f^{i-j} \paren{x} < y$, and so $x' \leq f^{i} \paren{x} = f^{j} \paren{f^{i-j} \paren{x}} < f^{j} \paren{y} \leq y'$. It follows that $\orb{f}{x} <_{f}^{s} \orb{f}{y}$.
        \item (b)(i) implies (b)(ii):\\
        Suppose $x' \sim_{f} x$ and $y' \sim_{f} y$ such that $x' \leq y'$; by definition of $\sim_{f}$, there are $i, j \in \Z$ such that $f^{i} \paren{x} \leq x'$, and $y' \leq f^{j} \paren{y}$. Then by transitivity, $f^{i} \paren{x} \leq f^{j} \paren{y}$, and so $f^{i-j} \paren{x} \leq y$.
        \item (ii) iff (iii) for both (a) and (b):\\
        Immediate from $f^{n} \paren{x} < y \iff x < f^{-n} \paren{y}$.
        \item (b)(iii) iff (b)(iv):\\
        Immediate from the definitions.
        \item (a)(iii) implies (a)(iv):\\
        For all $n \in \Z$, $x < f^{n} \paren{y} \implies b_{n}^{f} \paren{x,y}$. Moreover, the case $n = 0$ implies $x < y$, so $x \neq y$.
        \item (a)(iv) implies (a)(iii):\\
        For all $n \in \Z$, $b_{n}^{f} \paren{x,y} \implies x \leq f^{n} \paren{y}$. Suppose for contradiction $x = f^{n} \paren{y}$ for some $n \in \Z$. Then in particular, $x \sim_{f} y$. If $\parity{x,f} = 0$, then $b_{0}^{f} \paren{x,y} \iff x \leq y \implies x = y$, contradicting $x \neq y$. If $\parity{x,f} \neq 0$, then there is some $k \in \Z$ such that $y < f^{k} \paren{x} \implies \neg b_{-k}^{f} \paren{x,y}$, contradicting $\mathbf{b}^{f} \paren{x,y} = \mathbf{1}$. Hence $x \leq f^{n} \paren{y}$ for every $n \in \Z$.
    \end{itemize}
\end{proof}

\begin{cor}
    Suppose $p \in \mathcal{P}$, and $f, g \in \brack{p}$ are isomorphic over $\paren{x,y}$. Then $\orb{f}{x} <_{f}^{s} \orb{f}{y} \iff \orb{g}{x} <_{g}^{s} \orb{g}{y}$, and $\orb{f}{x} \leq_{f}^{w} \orb{f}{y} \iff \orb{g}{x} \leq_{g}^{w} \orb{g}{y}$.
\end{cor}

\subsection{Useful combinatorial criteria}\label{subsec:UsefulCombinatorialCriteria}

Analyzing even finite $L$-structures will require us to consider the infinite $\mathbf{b}$-sequences. Fortunately, there are a number of finitary conditions which allow us to control these sequences in various ways. We collect these methods here.

\begin{lem}\label{OrbitalIsDirectedEquivs}
    For $f \in \Aut{P}$ and $x \in P$ with $\parity{x,f} = +1$ (resp. $-1$), the following are equivalent:
    \begin{enumerate}[label=(\roman*)]
        \item\label{OrbitalUpperDirected} $\orb{f}{x}$ is upper-directed.
        \item\label{OrbitalLowerDirected} $\orb{f}{x}$ is lower-directed.
        \item\label{aSeqEventuallyConstant} The sequence $\mathbf{b}_{\geq 0}^{f} \paren{x,x}$ (resp. $\mathbf{b}_{\leq 0}^{f} \paren{x,x}$) is eventually constant $1$.
        \item\label{bxyEventuallyConst} The sequence $\mathbf{b}^{f} \paren{x,y}$ is eventually constant on both sides for every $y \in P$.
        \item\label{byxEventuallyConst} The sequence $\mathbf{b}^{f} \paren{y,x}$ is eventually constant on both sides for every $y \in P$.
    \end{enumerate}
\end{lem}
\begin{proof}
We assume $\parity{x,f} = +1$, as the proof for the case $\parity{x,f} = -1$ is symmetric.
    \begin{itemize}
        \item \ref{OrbitalUpperDirected} implies \ref{aSeqEventuallyConstant}:\\
        Let $N := \sp{x,f}$. By \ref{OrbitalUpperDirected}, take $z \in \orb{f}{x}$ to be an upper bound for the set $\brace{f^{i} \paren{x} : 0 \leq i < N}$. Since $z \sim_{f} x$, there is some $j \in \Z$ such that $z \leq f^{j} \paren{x}$. Thus for $1 \leq i \leq N$ we have $f^{N-i} \paren{x} \leq z \leq f^{j} \paren{x}$ implies $x < f^{N} \paren{x} \leq f^{i} \paren{z} \leq f^{j+i} \paren{x}$, and so $b_{j+i}^{f} \paren{x,x} = 1$. Since $b_{i}^{f} \leq b_{i + N}^{f}$ for all $i$, it follows that $\mathbf{b}_{\geq 1}^{f}$ is eventually constant $1$.
        \item \ref{aSeqEventuallyConstant} implies \ref{OrbitalUpperDirected}:\\
        It suffices to show any two elements have an upper bound. Suppose $\mathbf{b}_{\geq 1}^{f} \paren{x,x}$ is eventually constant $1$, and let $x',x'' \in \orb{f}{x}$. Take $j_{1}, j_{2} \in \Z$ such that $x' \leq f^{j_{1}} \paren{x}$ and $x'' \leq f^{j_{2}} \paren{x}$. By hypothesis, there is $k$ large enough so that both $x \leq f^{k - j_{1}} \paren{x}$ and $x \leq f^{k - j_{2}} \paren{x}$. Then $f^{k} \paren{x}$ is an upper bound of $x'$ and $x''$: indeed, $x' \leq f^{j_{1}} \paren{x} \leq f^{j_{1}} \paren{f^{k - j_{1}} \paren{x}} = f^{k} \paren{x}$, and $x'' \leq f^{j_{2}} \paren{x} \leq f^{j_{2}} \paren{f^{k - j_{2}} \paren{x}} = f^{k} \paren{x}$.
        \item \ref{OrbitalLowerDirected} implies \ref{aSeqEventuallyConstant}:\\
        Similar to the upper-directed case: let $z \in \orb{f}{x}$ be a lower bound of the set $\brace{f^{i} \paren{x} : 0 \leq i < \sp{x,f}}$. Then choose $j \in \Z$ such that $f^{j} \paren{x} \leq z$, and so $x \leq f^{-j} \paren{z} \leq f^{i-j} \paren{x}$ for each $0 \leq i < \sp{x,f}$. Combined with $b_{i}^{f} \leq b_{i+\sp{x,f}}^{f}$, it follows that $\mathbf{b}_{\geq 1}^{f}$ is eventually constant $1$.
        \item \ref{aSeqEventuallyConstant} implies \ref{OrbitalLowerDirected}:\\
        Similar to the upper bound case: suppose $\mathbf{b}_{\geq 1}^{f} \paren{x,x}$ is eventually constant $1$, and let $x',x'' \in \orb{f}{x}$. Take $j_{1}, j_{2} \in \Z$ such that $f^{j_{1}} \paren{x} \leq x'$ and $f^{j_{2}} \paren{x} \leq x''$. Then there is $k$ large enough so that $f^{-k} \paren{x} \leq f^{j_{1}} \paren{x} \leq x'$ and $f^{-k} \paren{x} \leq f^{j_{2}} \paren{x} \leq x''$.
        \item \ref{aSeqEventuallyConstant} implies \ref{bxyEventuallyConst}:
        \begin{itemize}
            \item Case I: $\orb{f}{x} <_{f}^{s} \orb{f}{y}$\\
            Then $\mathbf{b}^{f} \paren{x,y}$ is constant $1$ on both sides.
            \item Case II: $\orb{f}{x} \nleq_{f}^{w} \orb{f}{y}$\\
            Then $\mathbf{b}^{f} \paren{x,y}$ is constant $0$ on both sides.
            \item Case III: $\orb{f}{x} \leq_{f}^{w} \orb{f}{y}$ and $\orb{f}{x} \nless_{f}^{s} \orb{f}{y}$\\
            By \ref{aSeqEventuallyConstant}, take $N$ large enough so that $i \geq N$ implies $b_{i}^{f} \paren{x,x}$ holds. By $\orb{f}{x} \leq_{f}^{w} \orb{f}{y}$, take $j \in \Z$ such that $b_{j}^{f} \paren{x,y}$ holds. By a previous remark, for all $i \geq N$ we have $b_{i}^{f} \paren{x,x} \wedge b_{j}^{f} \paren{x,y} \implies b_{i+j}^{f} \paren{x,y}$. That is, $b_{i}^{f} \paren{x,y}$ holds for all $i \geq N+j$.
        
            By $\orb{f}{x} \nless_{f}^{s} \orb{f}{y}$, take $j' \in \Z$ such that $\neg b_{j'}^{f} \paren{x,y}$. By the above, we have $b_{i}^{f} \paren{x,x} \wedge b_{j'-i}^{f} \paren{x,y} \implies b_{j'}^{f} \paren{x,y}$, and so by the contrapositive, $\neg b_{j'}^{f} \paren{x,y} \implies \neg b_{j'-i}^{f} \paren{x,y}$ for all $i \geq N$. That is, $b_{i}^{f} \paren{x,y}$ fails for all $i \leq j' - N$.
        \end{itemize}
    
        (The case for $\parity{x,f} = -1$ follows by the same argument.)
        \item \ref{bxyEventuallyConst} implies \ref{aSeqEventuallyConstant}:\\
        Immediate.
        \item \ref{aSeqEventuallyConstant} iff \ref{byxEventuallyConst}:\\
        Similar to the proof of \ref{aSeqEventuallyConstant} iff \ref{bxyEventuallyConst}.
    \end{itemize}
\end{proof}

\begin{rmk}
    We provide some illustrations which may assist with the intuition. Suppose $f \in \Aut{P}$ and $x \in P$ with $\parity{x,f} = +1$ and $\sp{x,f} = 4$. Since $f$ is an automorphism, this implies $f^{k} \paren{x} < f^{k+4} \paren{x}$ for every $k \in \Z$. Thus, this partitions the orbit $f^{\Z} \paren{x}$ into $4$ ``rails'', as shown in \Cref{fig:FourRails}.
    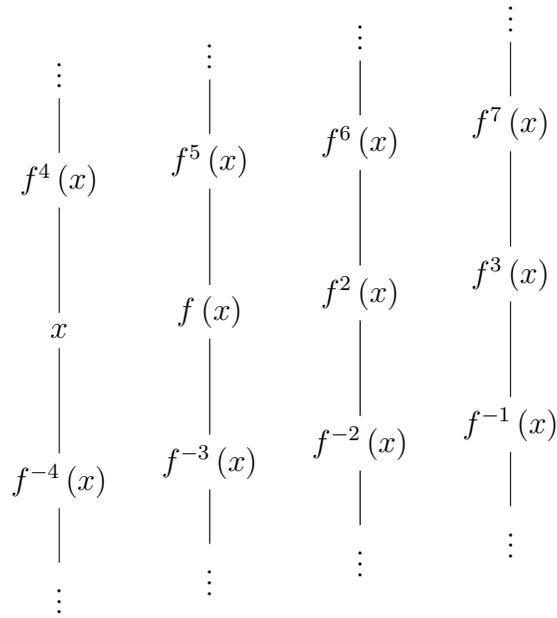
\begin{figure}[h!]
        \centering
        \begin{tikzpicture}
            \node (x) at (0,0) {$x$};
            \node (fx) at (2,0.25) {$f \paren{x}$};
            \node (f2x) at (4,0.5) {$f^{2} \paren{x}$};
            \node (f3x) at (6,0.75) {$f^{3} \paren{x}$};
            \node (f4x) at (0,2) {$f^{4} \paren{x}$};
            \node (f5x) at (2,2.25) {$f^{5} \paren{x}$};
            \node (f6x) at (4,2.5) {$f^{6} \paren{x}$};
            \node (f7x) at (6,2.75) {$f^{7} \paren{x}$};
            \node (f8x) at (0,3.5) {$\vdots$};
            \node (f9x) at (2,3.75) {$\vdots$};
            \node (f10x) at (4,4) {$\vdots$};
            \node (f11x) at (6,4.25) {$\vdots$};
            \node (f-1x) at (6,-1.25) {$f^{-1} \paren{x}$};
            \node (f-2x) at (4,-1.5) {$f^{-2} \paren{x}$};
            \node (f-3x) at (2,-1.75) {$f^{-3} \paren{x}$};
            \node (f-4x) at (0,-2) {$f^{-4} \paren{x}$};
            \node (f-5x) at (6,-2.75) {$\vdots$};
            \node (f-6x) at (4,-3) {$\vdots$};
            \node (f-7x) at (2,-3.25) {$\vdots$};
            \node (f-8x) at (0,-3.5) {$\vdots$};
            \draw (f-8x) -- (f-4x) -- (x) -- (f4x) -- (f8x);
            \draw (f-7x) -- (f-3x) -- (fx) -- (f5x) -- (f9x);
            \draw (f-6x) -- (f-2x) -- (f2x) -- (f6x) -- (f10x);
            \draw (f-5x) -- (f-1x) -- (f3x) -- (f7x) -- (f11x);
        \end{tikzpicture}
        \caption{Four ``rails''.}
        \label{fig:FourRails}
    \end{figure}
    Of course, it may well be the case that $x$ is related to other points of its orbit also, not just the points $f^{4k} \paren{x}$. For the sake of example, suppose also that $x < f^{5} \paren{x}$. Since $f$ is an automorphism, this also implies relations such as $f \paren{x} < f^{6} \paren{x}$, $f^{-3} \paren{x} < f^{2} \paren{x}$, and so on. We add these relations in dotted red in \Cref{fig:FourRailsWithHops}.
    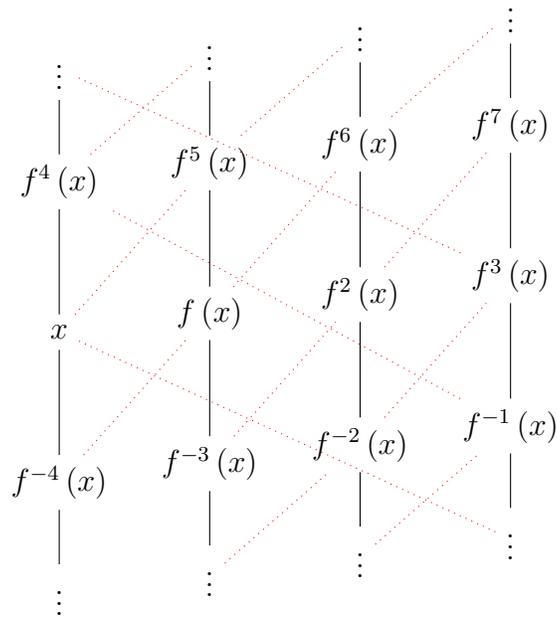
\begin{figure}[h!]
        \centering
        \begin{tikzpicture}
            \node (x) at (0,0) {$x$};
            \node (fx) at (2,0.25) {$f \paren{x}$};
            \node (f2x) at (4,0.5) {$f^{2} \paren{x}$};
            \node (f3x) at (6,0.75) {$f^{3} \paren{x}$};
            \node (f4x) at (0,2) {$f^{4} \paren{x}$};
            \node (f5x) at (2,2.25) {$f^{5} \paren{x}$};
            \node (f6x) at (4,2.5) {$f^{6} \paren{x}$};
            \node (f7x) at (6,2.75) {$f^{7} \paren{x}$};
            \node (f8x) at (0,3.5) {$\vdots$};
            \node (f9x) at (2,3.75) {$\vdots$};
            \node (f10x) at (4,4) {$\vdots$};
            \node (f11x) at (6,4.25) {$\vdots$};
            \node (f-1x) at (6,-1.25) {$f^{-1} \paren{x}$};
            \node (f-2x) at (4,-1.5) {$f^{-2} \paren{x}$};
            \node (f-3x) at (2,-1.75) {$f^{-3} \paren{x}$};
            \node (f-4x) at (0,-2) {$f^{-4} \paren{x}$};
            \node (f-5x) at (6,-2.75) {$\vdots$};
            \node (f-6x) at (4,-3) {$\vdots$};
            \node (f-7x) at (2,-3.25) {$\vdots$};
            \node (f-8x) at (0,-3.5) {$\vdots$};
            \draw (f-8x) -- (f-4x) -- (x) -- (f4x) -- (f8x);
            \draw (f-7x) -- (f-3x) -- (fx) -- (f5x) -- (f9x);
            \draw (f-6x) -- (f-2x) -- (f2x) -- (f6x) -- (f10x);
            \draw (f-5x) -- (f-1x) -- (f3x) -- (f7x) -- (f11x);
            \draw[red, dotted] (f-4x) -- (fx) -- (f6x) -- (f11x);
            \draw[red, dotted] (f-8x) (f-3x) -- (f2x) -- (f7x);
            \draw[red, dotted] (f-7x) -- (f-2x) -- (f3x) -- (f8x);
            \draw[red, dotted] (f-6x) -- (f-1x) -- (f4x) -- (f9x);
            \draw[red, dotted] (f-5x) -- (x) -- (f5x) -- (f10x);
        \end{tikzpicture}
        \caption{$b_{5}^{f} \paren{x,x}$ in dotted red.}
        \label{fig:FourRailsWithHops}
    \end{figure}
    These extra relations allow one to ``hop'' between rails. In this terminology, $x$ satisfies \Cref{OrbitalIsDirectedEquivs} iff one can hop from $x$ to any other rail. This holds in the example in \Cref{fig:FourRailsWithHops}, but it would fail if we instead chose $x < f^{6} \paren{x}$, since then $x$ could only hop to even-numbered rails.
\end{rmk}

\begin{lem}[Forcing $<^{s}$ by fixed points]\label{FixedPointBetweenForcesStronglyLess}
    Let $f \in \Aut{P}$, and let $x,y,c \in P$ such that $f \paren{c} = c$ and $x < c < y$. Then $\orb{f}{x} <_{f}^{s} \orb{f}{y}$.
\end{lem}
\begin{proof}
    Note $x < c = f^{n} \paren{c}$ for all $n \in \Z$, so $\orb{f}{x} <_{f}^{s} \orb{f}{c}$ by \Cref{OrbitalOrderWeakerCondition}. Similarly, $f^{n} \paren{c} = c < y$ for all $n \in \Z$, so $\orb{f}{c} <_{f}^{s} \orb{f}{y}$. Then $\orb{f}{x} <_{f}^{s} \orb{f}{y}$ by transitivity.
\end{proof}

\begin{lem}[Preventing $\leq^{w}$ by fixed points]\label{NLemma}
    Let $f \in \Aut{P}$, and let $x,y \in P$ such that $x \nleq y$. Suppose one of the following holds:
    \begin{enumerate}[label=(\roman*)]
        \item There is $c \in P$ with $c < x$, $c \perp y$, and $f \paren{c} = c$.
        \item There is $d \in P$ with $y < d$, $d \perp x$, and $f \paren{d} = d$.
    \end{enumerate}
    (See \Cref{fig:NExample}.) Then $\orb{f}{x} \nleq_{f}^{w} \orb{f}{y}$.
\end{lem}
\begin{figure}[h!]
    \centering
    \begin{tikzpicture}
        \node (x) at (0,0) {$x$};
        \node (c) at (0,-2) {$c$};
        \node (y) at (2,-1) {$y$};
        \node (d) at (2,1) {$d$};
        \draw (c) -- (x);
        \draw (y) -- (d);
        \draw[dotted] (x) -- (y);
    \end{tikzpicture}
    \caption{The dotted line indicates that either $y < x$ or $y \perp x$ is allowed.}
    \label{fig:NExample}
\end{figure}
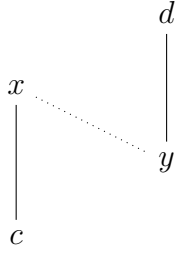
\begin{proof}
    Assume (i) holds. Then $\orb{f}{c} \leq_{f}^{w} \orb{f}{x}$. Suppose for contradiction $\orb{f}{x} \leq_{f}^{w} \orb{f}{y}$. Thus $\orb{f}{c} \leq_{f}^{w} \orb{f}{y}$ by transitivity, so by \Cref{OrbitalOrderWeakerCondition}, there exists $n \in \Z$ such that $f^{n} \paren{c} \leq y$. But $c$ is a fixed point, so $c \leq y$, contradicting that $c \perp y$. Hence $\orb{f}{x} \nleq_{f}^{w} \orb{f}{y}$. The case when (ii) holds is similar.
\end{proof}

The following criterion was used in \cite{KT01} to show by example that points of infinite spiral length cannot be ignored, as they can be forced to occur by finite configurations. We make extensive use of this idea, since forcing these points to occur will be not only unavoidable, but desirable.
\begin{lem}[``M'' configurations]\label{MLemma}
    Let $f \in \Aut{P}$, and let $x,y,z \in P$ such that $x < z$, $x < f \paren{y}$, $f \paren{z} < z$, $y < f \paren{y}$, $x \perp f \paren{z}$, and $x \perp y$. (See \Cref{fig:MExample}.) Then $\sp{x,f} = \infty$, i.e.\ the orbit $f^{\Z} \paren{x}$ is an infinite antichain.
\end{lem}
\begin{figure}[h!]
    \centering
    \begin{tikzpicture}
        \node (fy) at (-1,1) {$f \paren{y}$};
        \node (y) at (-2,0) {$y$};
        \node (x) at (0,0) {$x$};
        \node (z) at (1,1) {$z$};
        \node (fz) at (2,0) {$f \paren{z}$};
        \draw (y) -- (fy) -- (x) -- (z) -- (fz);
    \end{tikzpicture}
    \caption{An ``M''.}
    \label{fig:MExample}
\end{figure}
\begin{proof}
    Note that $x \parallel f^{n} \paren{x}$ iff $f^{-n} \paren{x} \parallel x$, so it suffices to show $x \perp f^{n} \paren{x}$ for all $n \geq 1$. Indeed, if $x \leq f^{n} \paren{x}$, then $x < z$ implies $f \paren{z} \geq f^{n} \paren{z} > f^{n} \paren{x} \geq x$, contradicting $x \perp f \paren{z}$. If instead $f^{n} \paren{x} \leq x$, then $x < f \paren{y}$ implies $f^{n} \paren{x} < f \paren{y} \leq f^{n} \paren{y}$, so $x < y$, contradicting $x \perp y$.
\end{proof}

The following technique was also introduced by \cite{KT01}, as a way to control the order relationship between orbits when at least one has infinite spiral length. The idea is to add points of finite spiral length above and below a point with infinite spiral length (``tightening'' around it, hence the name we will choose) to force periodicity of certain $\mathbf{b}$-sequences.

Adopting model-theoretic terminology, we let $\operatorname{qftp}_{<} \paren{x,F}$ denote the complete quantifier-free $1$-type of $x$ over the set of parameters $F$. That is, $\operatorname{qftp}_{<} \paren{x,F} = \operatorname{qftp}_{<} \paren{x', F}$ if and only if $x < y \iff x' < y$ and $y < x \iff y < x'$ for all $y \in F$.

\begin{lem}\label{InfiniteAntichainTypes}
    Let $f \in \Aut{P}$, $x \in P$, and $F \subseteq P$ finite such that $F \cap f^{\Z} \paren{x} = \emptyset$. Then there exists $k \in \Z$ such that for infinitely many $n \neq 0$, we have $\operatorname{qftp}_{<} \paren{f^{k} \paren{x},F} = \operatorname{qftp}_{<} \paren{f^{k+n} \paren{x}, F}$. Moreover, we may assume $k$ has whichever sign, and is as large in absolute value, as we want.
\end{lem}
\begin{proof}
    Essentially by Pigeonhole Principle: since $F$ is finite and the language $\brace{<}$ is finite, there are finitely many complete quantifier-free $1$-types over $F$, so one of them must appear in the list $\brace{\operatorname{qftp}_{<} \paren{f^{k} \paren{x},F} : k \in \Z}$ infinitely often. For the ``moreover'' part, we apply this argument to the sets $\brace{\operatorname{qftp}_{<} \paren{f^{k} \paren{x},F} : k \leq 0}$ and $\brace{\operatorname{qftp}_{<} \paren{f^{k} \paren{x},F} : k \geq 0}$ accordingly.
\end{proof}

\begin{lem}[Tightening spirals]\label{TighteningSpirals}
    Let $f \in \Aut{P}$, $x \in P$, and $F \subseteq P$ finite such that $f^{\Z} \paren{x} \cap F = \emptyset$. Let $n \geq 1$ (respectively, $n \leq -1$) such that $\operatorname{qftp}_{<} \paren{x, F} = \operatorname{qftp}_{<} \paren{f^{n} \paren{x}, F}$. Suppose there exist $\alpha,\beta \in P$ such that:
    \begin{itemize}
        \item $\parity{\alpha,f} = +1$ (respectively $-1$) and $\parity{\beta,f} = -1$ (respectively $+1$);
        \item $\sp{\alpha,f} = \sp{\beta,f} = n$;
        \item $\alpha < x < \beta$;
        \item For each $0 \leq i \leq n$ (respectively $n \leq i \leq 0$), $\operatorname{qftp} \paren{f^{i} \paren{\alpha}, F} = \operatorname{qftp} \paren{f^{i} \paren{\beta}, F} = \operatorname{qftp} \paren{f^{i} \paren{x}, F}$.
    \end{itemize}
    (See \Cref{fig:FlagExample}.) Then the sequences $\mathbf{b}_{\geq 0}^{f} \paren{y,x}$ and $\mathbf{b}_{\leq 0}^{f} \paren{x,y}$ (respectively $\mathbf{b}_{\leq 0}^{f} \paren{y,x}$ and $\mathbf{b}_{\geq 0}^{f} \paren{x,y}$) are $n$-periodic for all $y \in F$.
\end{lem}
\begin{figure}[h!]
    \centering
    \begin{tikzpicture}
        \node (x) at (0,0) {$x$};
        \node (fx) at (2,0) {$f \paren{x}$};
        \node (f2x) at (4,0) {$f^{2} \paren{x}$};
        \node (f3x) at (6,0) {$f^{3} \paren{x}$};
        \node (f4x) at (8,0) {$f^{4} \paren{x}$};
        \node (a) at (0,-3) {$\alpha$};
        \node (fa) at (2,-2.8) {$f \paren{\alpha}$};
        \node (f2a) at (4,-2.6) {$f^{2} \paren{\alpha}$};
        \node (f3a) at (6,-2.4) {$f^{3} \paren{\alpha}$};
        \node (f4a) at (8,-1.2) {$f^{4} \paren{\alpha}$};
        \node (b) at (0,3) {$\beta$};
        \node (fb) at (2,2.8) {$f \paren{\beta}$};
        \node (f2b) at (4,2.6) {$f^{2} \paren{\beta}$};
        \node (f3b) at (6,2.4) {$f^{3} \paren{\beta}$};
        \node (f4b) at (8,1.2) {$f^{4} \paren{\beta}$};
        \node (y) at (4.5,-5) {$y$};
        \draw (y) -- (a) -- (x) -- (b);
        \draw (f4b) -- (f4x) -- (f4a);
        \draw (b) to[out=-20,in=180] (f4b);
        \draw (a) to[out=20,in=180] (f4a);
        \draw (y) -- (fa) -- (fx) -- (fb);
        \draw (f2a) -- (f2x) -- (f2b);
        \draw (f3a) -- (f3x) -- (f3b);
        \draw (y) -- (f4a);
    \end{tikzpicture}
    \caption{A pair of tightening spirals with $n = 4$, $y < x$, $y < f \paren{x}$, and $y < f^{4} \paren{x}$.}
    \label{fig:FlagExample}
\end{figure}
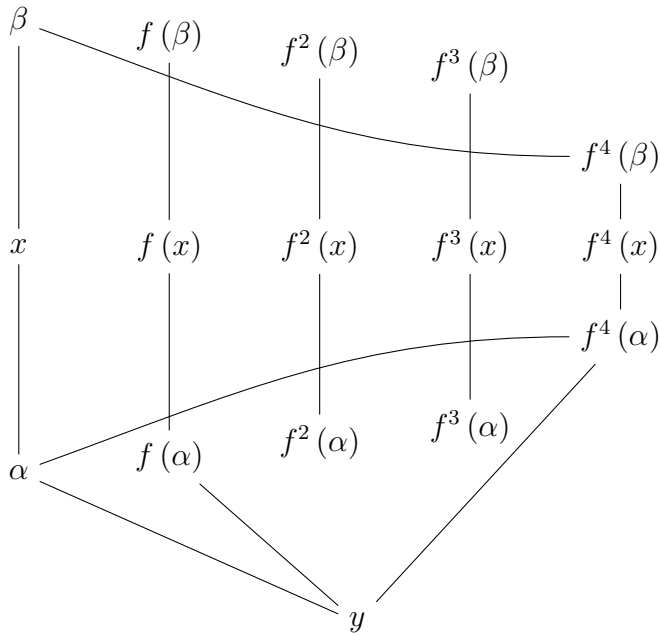
\begin{proof}
    We leave the ``respective'' cases as an exercise, as they are similar. For each $0 \leq i < n$ and $k \geq 0$, we have:
    \begin{align*}
        b_{i}^{f} \paren{y,x} &\implies y < f^{i} \paren{x}\\
        &\implies y < f^{i} \paren{\alpha} < f^{i + kn} \paren{\alpha} < f^{i + kn} \paren{x}\\
        &\implies b_{i+kn}^{f} \paren{y,x}\\
        &\implies y < f^{i+kn} \paren{x} < f^{i+kn} \paren{\beta} < f^{i} \paren{\beta}\\
        &\implies y < f^{i} \paren{x}\\
        &\implies b_{i}^{f} \paren{y,x}
    \end{align*}
    Hence $b_{i}^{f} \paren{y,x} \iff b_{i+kn}^{f} \paren{y,x}$ for all $k \geq 0$; thus, $\mathbf{b}_{\geq 0}^{f} \paren{y,x}$ is $n$-periodic. Similarly:
    \begin{align*}
        b_{-i}^{f} \paren{x,y} &\implies y > f^{i} \paren{x}\\
        &\implies y > f^{i} \paren{\beta} > f^{i + kn} \paren{\beta} > f^{i+kn} \paren{x}\\
        &\implies b_{-i-kn}^{f} \paren{x,y}\\
        &\implies y > f^{i+kn} \paren{x} > f^{i+kn} \paren{\alpha} > f^{i} \paren{\alpha}\\
        &\implies y > f^{i} \paren{x}\\
        &\implies b_{-i}^{f} \paren{x,y}
    \end{align*}
    Hence $b_{-i}^{f} \paren{x,y} \iff b_{-i-kn}^{f} \paren{x,y}$ for all $k \geq 0$; thus, $\mathbf{b}_{\leq 0}^{f} \paren{x,y}$ is $n$-periodic.
\end{proof}

\section{The random poset}\label{sec:TheRandomPoset}

\subsection{Generic automorphisms of \Fraisse{} limits}

We begin by reviewing the needed background. Let $\mathfrak{K}$ be a \Fraisse{} class of finite structures in a fixed relational language. Then the \Fraisse{} limit is the unique (up to isomorphism) countable structure $\mathbf{K}$ satisfying:
\begin{itemize}
    \item \textit{Universality}: Up to isomorphism, $\mathfrak{K}$ is precisely the class of finite substructures of $\mathbf{K}$;
    \item \textit{Ultrahomogeneity}\footnote{This notion differs from the usual notion of homogeneity in model theory, since here we do not assume elementarity for the partial maps.}: Every finite partial automorphism $p \from \mathbf{K} \rightharpoonup \mathbf{K}$ extends to a full automorphism $f \in \Aut{\mathbf{K}}$. In other words, $\brack{p}$ is non-empty for all such $p$.
\end{itemize}

We equip the group $\Aut{\mathbf{K}}$ with the pointwise-convergence topology, whose basic open sets are of the form $\brack{p}$ for $p \in \mathcal{K}$. Since $\mathbf{K}$ is countable, $\Aut{\mathbf{K}}$ is a Polish group.

We describe in more detail the results on existence of generics alluded to in the introduction. The relevant properties for us will be the following:
\begin{defn}
    Let $\mathbf{K}$ be the limit of a \Fraisse{} class $\mathfrak{K}$, and let $\mathcal{K}$ denote the set of finite partial automorphisms of $\mathbf{K}$.
    \begin{itemize}
        \item Let $p,q \in \mathcal{K}$. An \textit{embedding} of $p$ into $q$ is an element $r \in \mathcal{K}$ such that $\dom{r} = \dur{p}$ and $r \circ p \circ r^{-1} \subseteq q$.
        \item A subset $\mathcal{L} \subseteq \mathcal{K}$ has the \textit{joint embedding property} (JEP) if for any $p_{1}, p_{2} \in \mathcal{L}$, there exist $q \in \mathcal{L}$ and embeddings $r_{1}, r_{2} \in \mathcal{K}$ of $p_{1}, p_{2}$ respectively into $q$.
        \item A subset $\mathcal{L} \subseteq \mathcal{K}$ has the \textit{amalgamation property} (AP) if for any $p, p_{1}, p_{2} \in \mathcal{L}$ and any embeddings $r_{1}, r_{2} \in \mathcal{K}$ of $p$ into $p_{1}, p_{2}$ respectively, there exists $q \in \mathcal{L}$ and embeddings $s_{1}, s_{2} \in \mathcal{K}$ of $p_{1}, p_{2}$ respectively into $q$ such that $s_{1} \circ r_{1} = s_{2} \circ r_{2}$.
        \item A subset $\mathcal{L} \subseteq \mathcal{K}$ is \textit{closed under conjugacy} if for all $p \in \mathcal{L}$ and $g \in \Aut{\mathbf{P}}$, we have $g \circ p \circ g^{-1} \in \mathcal{L}$.
        \item We say $\mathcal{K}$ has the \textit{cofinal amalgamation property} (CAP) if it admits a subset $\mathcal{L}$ which is closed under conjugacy, cofinal with respect to $\subseteq$, and has the amalgamation property.
    \end{itemize}
\end{defn}

\begin{thm}[Truss, \cite{Tru92}]\label{CAPJEPImpliesGenericsExist}
    Let $\mathbf{K}$ be a \Fraisse{} limit for which $\mathcal{K}$ has the CAP and the JEP\footnote{JEP is equivalent to amalgamation over the empty map, so if your witness to the CAP includes the empty map, this assumption is redundant.}. Then $\Aut{\mathbf{K}}$ admits generic elements.
\end{thm}
We remark that this gives sufficient conditions for existence of generics, not a characterization. The condition becomes necessary and sufficient if one replaces the CAP with the so-called \textit{weak amalgamation property}, proven in \cite{KR04} using Hjorth's theory of turbulence.

The proof of \Cref{CAPJEPImpliesGenericsExist} proceeds as follows: let $\mathcal{L} \subseteq \mathcal{K}$ witness the CAP. Then define, for each $p \in \mathcal{K}$:
$$D \paren{p} := \brace{f \in \Aut{\mathbf{K}} : p \subseteq f \implies \exists q \in \mathcal{L} \paren{p \subseteq q \subseteq f}},$$
and for each $p,p' \in \mathcal{L}$ with $p \subseteq p'$:
$$E \paren{p,p'} := \brace{f \in \Aut{\mathbf{K}} : p \subseteq f \implies \exists r \in \mathcal{K} 
\begin{pmatrix}
    \textnormal{$\dom{r} = \dur{p'}$, $r$ fixes $\dur{p}$}\\ 
    \textnormal{pointwise, and $r \circ p' \circ r^{-1} \subseteq f$}
\end{pmatrix}}.$$
Then each $D \paren{p}$ and each $E \paren{p,p'}$ are dense open, and the dense $G_{\delta}$ set $$\Gamma := \bigcap_{p \in \mathcal{K}} D \paren{p} \cap \bigcap_{\substack{p,p' \in \mathcal{L}\\p \subseteq p'}} E \paren{p,p'}$$ is a single conjugacy class.

\subsection{Random poset}

Throughout, we let $\paren{\mathbf{P},\leq}$ denote the \textit{random poset} -- that is, the \Fraisse{} limit of the class of finite partial orders. As in the previous section, $\mathcal{P}$ denotes the collection of finite partial automorphisms of $\mathbf{P}$.

\begin{rmk}
    Let $T$ be the theory of $\mathbf{P}$ in the language of partial orders. Then $T$ is $\omega$-categorical and admits quantifier elimination (see e.g. \cite{Mac11}, Corollary 3.1.3 and Proposition 3.1.6).
\end{rmk}

The following statement is a trivial consequence of universality and ultrahomogeneity of $\mathbf{P}$, but is fundamental to any argument about $\mathbf{P}$. Following \cite{KT01}, we call this the \textit{universal-homogeneity} property of $\mathbf{P}$.

\begin{lem}[Universal-homogeneity]\label{UniversalHomogeneity}
    For any $A \subseteq \mathbf{P}$ finite, and any finite poset $B$ such that $A \subseteq B$, there is an embedding $\phi \from B \hookrightarrow \mathbf{P}$ which is the identity on $A$.
\end{lem}

The motivation for our work begins with the theorem that $\Aut{\mathbf{P}}$ admits generics.

\begin{thm}[Kuske--Truss, \cite{KT01}]
    The family of determined partial automorphisms witnesses $\mathbf{P}$ having the CAP, and so $\Aut{\mathbf{P}}$ has generic elements.
\end{thm}

Our goal is a concrete description of the generic automorphism of the random poset $\mathbf{P}$ --- ideally, one which explicitly mentions finite partial automorphisms as little as possible. This is why we recalled the proof of \Cref{CAPJEPImpliesGenericsExist}: we will identify a list of conditions on $f \in \Aut{\mathbf{P}}$ which are necessary and sufficient for $f$ to lie in the $D \paren{p}$'s and $E \paren{p,p'}$'s defined above.

As an example of how this will be used, we begin with the following relation between orbital orders and determined partial automorphisms.
\begin{lem}\label{OrbitalOrderCapturedByDeterm}
    Let $p \in \mathcal{P}$ be determined, and let $x,y \in \dur{p}$. Then for the generic $f$ which extends $p$:
    \begin{enumerate}[label=(\roman*)]
        \item $\orb{f}{x} <_{f}^{s} \orb{f}{y}$ if and only if $x' < y'$ for every $x' \in \orb{f}{x} \cap \dur{p}$ and $y' \in \orb{f}{y} \cap \dur{p}$;
        \item $\orb{f}{x} \leq_{f}^{w} \orb{f}{y}$ if and only if $x' \leq y'$ for some $x' \in \orb{f}{x} \cap \dur{p}$ and $y' \in \orb{f}{y} \cap \dur{p}$.
    \end{enumerate}
\end{lem}
\begin{proof}
    \begin{enumerate}[label=(\roman*)]
        \item The forward direction is immediate. For the converse, define the poset $A := \dur{p} \sqcup \brace{a}$ by letting the immediate predecessors of $a$ be the maximal elements of $\orb{f}{x} \cap \dur{p}$, and letting the immediate successors of $a$ be the minimal elements of $\orb{f}{y} \cap \dur{p}$. This does not create any new relations between elements of $\dur{p}$ since by hypothesis, $x' < y'$ for all $x',y' \in \dur{p}$ such that $x' \sim_{f} x$ and $y' \sim_{f} y$. By \Cref{UniversalHomogeneity}, there is an embedding $A \hookrightarrow \mathbf{P}$ which is the identity on $\dur{p}$. Identify $a$ with its image under this embedding.
        
        Moreover, $z < a \iff p \paren{z} < a$ and $a < z \iff a < p \paren{z}$ for all $z \in \dom{p}$. Thus, $p \cup \brace{\paren{a,a}} \in \mathcal{P}$, and so admits a determined extension $q$. By $f \in E \paren{p,q}$, there is an $r \in \mathcal{P}$ with $\dom{r} = \dur{q}$, $r$ fixes $\dur{p}$ pointwise, and $r \circ q \circ r^{-1} \subseteq f$. Then $f \paren{r \paren{a}} = \paren{r \circ q \circ r^{-1}} \paren{r \paren{a}} = r \paren{q \paren{a}} = r \paren{a}$. Thus $\orb{f}{r\paren{a}} = \brace{r \paren{a}}$. Moreover, we have $x < a < y \implies x = r \paren{x} < r \paren{a} < r \paren{y} = y$, so $\orb{f}{x} <_{f} \orb{f}{y}$ by \Cref{FixedPointBetweenForcesStronglyLess}.

        \item The reverse direction is immediate. For the converse, suppose $x' \nleq y'$ for all $x' \in \orb{f}{x} \cap \dur{p}$ and $y' \in \orb{f}{y} \cap \dur{p}$. Then define the poset $A := \dur{p} \sqcup \brace{a}$ where $a$ is minimal in $A$, and $a < z$ iff there is $x' \in \orb{f}{x} \cap \dur{p}$ with $x' \leq z$. Embed $A \hookrightarrow \mathbf{P}$ and identify $a$ with its image under this embedding.
        
        Then $a < p \paren{z} \iff a < z$ for all $z \in \dom{p}$; thus, $p \cup \brace{\paren{a,a}} \in \mathcal{P}$, and so admits a determined extension $q$. By $f \in E \paren{p,q}$, there is an $r \in \mathcal{P}$ with $\dom{r} = \dur{q}$, $r$ fixes $\dur{p}$ pointwise, and $r \circ q \circ r^{-1} \subseteq f$. Then $f \paren{r \paren{a}} = r \paren{a}$ is fixed, $r \paren{a} < r \paren{x} = x$, and $a \perp y \iff r \paren{a} \perp r \paren{y} = y$; thus $\orb{f}{x} \nleq_{f}^{w} \orb{f}{y}$ by \Cref{NLemma}.
    \end{enumerate}
\end{proof}

\begin{lem}\label{ConvexSubsetRandomIffDirected}
    Let $A \subseteq \mathbf{P}$ be convex. Then $A \isom \mathbf{P}$ if and only if $A$ is both upper- and lower-directed.
\end{lem}
\begin{proof}
    The forward implication is immediate; for the converse, we employ a back-and-forth argument. ``Back'' is immediate from universal-homogeneity, so we only show ``forth''. Suppose $\gamma \from \mathbf{P} \rightharpoonup A$ is a finite partial isomorphism and let $z \in \mathbf{P} \setminus \dom{\gamma}$. Let $\paren{x_{i}}_{i<n} \subseteq \dom{\gamma}$ be the immediate predecessors of $z$ in the finite poset $\dom{\gamma} \cup \brace{z}$, and let $\paren{y_{j}}_{j<m} \subseteq \dom{\gamma}$ be the immediate successors of $z$.
    \begin{itemize}
        \item Case I: $z$ is neither maximal nor minimal in $\dom{\gamma} \cup \brace{z}$ (that is, both $n \neq 0$ and $m \neq 0$)\\
        By universal homogeneity, let $c \in \mathbf{P}$ such that the $\gamma \paren{x_{i}}$'s are its immediate predecessors and the $\gamma \paren{y_{j}}$'s are its immediate successors. Then $\gamma' := \gamma \cup \brace{\paren{z,c}}$ is a partial isomorphism. By convexity, $\gamma \paren{x_{0}} < c < \gamma \paren{y_{0}}$ implies $c \in A$, and so $\gamma' \from \mathbf{P} \rightharpoonup A$ as desired.
        \item Case II: $z$ is maximal but not minimal in $\dom{\gamma} \cup \brace{z}$\\
        Let $b$ be an upper bound in $A$ for $\ran{\gamma}$. By universal homogeneity, let $c \in \mathbf{P}$ such that the $\gamma \paren{x_{i}}$'s are its immediate predecessors and $b$ is its only immediate successor. Then $\gamma \paren{x_{0}} < c < b$ implies $c \in A$, and so $\gamma' := \gamma \cup \brace{\paren{z,c}}$ is a partial isomorphism $\mathbf{P} \rightharpoonup A$.
        \item Case III: $z$ is minimal but not maximal in $\dom{\gamma} \cup \brace{z}$\\
        Similar to Case II. Let $a \in A$ be a lower bound of $\ran{\gamma}$, and let $c$ be above only $a$ and below all the $\gamma \paren{y_{j}}$'s.
        \item Case IV: $z$ is both maximal and minimal in $\dom{\gamma} \cup \brace{z}$\\
        Let $a$ and $b$ be lower and upper bounds respectively in $A$ of $\ran{\gamma}$, and let $c \in \mathbf{P}$ incomparable to $\ran{\gamma}$ and $a < c < b$. Then $\gamma' := \gamma \cup \brace{\paren{z,c}}$ is as desired.
    \end{itemize}
    This completes ``forth'', so $A \isom \mathbf{P}$ as desired.
\end{proof}

Since orbitals are convex, we immediately have the following.

\begin{cor}
    For $f \in \Aut{\mathbf{P}}$ and $x \in \mathbf{P}$ with $\parity{x,f} \neq 0$, we have $\orb{f}{x} \isom \mathbf{P}$ iff the equivalent conditions of \Cref{OrbitalIsDirectedEquivs} hold.
\end{cor}

\subsection{Characterizing generic automorphisms}\label{subsec:CharacterizingGenericAutomorphisms}

Here we give a characterization of the generic conjugacy class $\Gamma$.

\begin{defn}
    Let $\Sigma$ be the set of all $f \in \Aut{\mathbf{P}}$ satisfying the following:
    \begin{enumerate}[label=(\Alph*)]
        \item\label{ClassIsDense} $f$ has dense conjugacy class.
        \item\label{Ultrahomogeneity} The structure $\mathbf{P}_{f}$ is ultrahomogeneous.
        
        \item\label{OrbitalIsRandom} For all $x \in \mathbf{P}$ with $\parity{x,f} \neq 0$, the equivalent conditions of \Cref{OrbitalIsDirectedEquivs} hold. (That is, $\mathbf{b}^{f} \paren{x,y}$ and $\mathbf{b}^{f} \paren{y,x}$ are eventually constant on both sides for all $y \in \mathbf{P}$. By the corollary to \Cref{ConvexSubsetRandomIffDirected}, this is also equivalent to $\orb{f}{x} \isom \mathbf{P}$.)
        \item\label{CanFitTighteningSpirals} For all $x,y \in \mathbf{P}$ and $\sigma \in \brace{-1,+1}$ such that $\sp{x,f} = \infty$ and $y \notin f^{\Z} \paren{x}$, then there exist $k \in \Z$, $n \geq 1$, and $\alpha, \beta, \in \mathbf{P}$ satisfying:
        \begin{itemize}
            \item $\parity{\alpha,f} = -\parity{\beta,f} = \sigma$;
            \item $\sp{\alpha,f} = \sp{\beta,f} = n$;
            \item $\alpha < f^{k} \paren{x} < \beta$;
            \item For each $0 \leq i \leq n$, $f^{\sigma i} \paren{\alpha}$, $f^{\sigma i} \paren{\beta}$, and $f^{k + \sigma i} \paren{x}$ have the same types over $\brace{y}$.
        \end{itemize}
    \end{enumerate}
\end{defn}

As the next lemma shows, density of the conjugacy class of $f$ implies a certain amount of universality for the structure $\mathbf{P}_{f}$:

\begin{lem}[Universality for $\mathbf{P}_{f}$]\label{DenseClassImpliesUniversality}
    Suppose $f \in \Aut{\mathbf{P}}$ has a dense conjugacy class. Then for every $p \in \mathcal{P}$ and every $g \in \brack{p} \cap D \paren{p}$, we have $\paren{\dur{p}, \mathbf{b}^{g}} \hookrightarrow \mathbf{P}_{f}$.
\end{lem}
\begin{proof}
    Take $q \in \mathcal{P}$ determined with $p \subseteq q \subseteq g$. Since $\brack{q}$ is open in $\Aut{\mathbf{P}}$, choose $f' \in \brack{q}$ which is conjugate to $f$. Then since $q$ is determined and $f',g$ extend $q$, $\paren{\dur{p}, \mathbf{b}^{g}} \subseteq \paren{\dur{q}, \mathbf{b}^{g}} \isom \paren{\dur{q}, \mathbf{b}^{f'}} \subseteq \mathbf{P}_{f'} \isom \mathbf{P}_{f}$.
\end{proof}

Recall that $\brack{p}$ is non-empty open for all $p$ by ultrahomogeneity of $\mathbf{P}$, and $D \paren{p}$ is dense open for all $p$, so the condition on $g$ is not vacuous. Thus, $\mathbf{P}_{f}$ admits a certain amount of universal-homogeneity as an $L$-structure, so we can guarantee certain desirable behavior in $\mathbf{P}_{f}$. We collect some critical properties here, which include converses to \Cref{FixedPointBetweenForcesStronglyLess,NLemma,MLemma}.

\begin{prop}\label{PropertiesOfSigma}
    Suppose $f \in \Sigma$.
    \begin{enumerate}
        \item\label{AllSequencesEventuallyPeriodic} For all $x,y \in \mathbf{P}$, the sequence $\mathbf{b}^{f} \paren{x,y}$ is eventually periodic on both sides.
        \item\label{CanFitMs} (Converse to \Cref{FixedPointBetweenForcesStronglyLess}) If $x \in \mathbf{P}$ with $\sp{x,f} = \infty$, then there exist $y,z \in \mathbf{P}$ such that $y < f \paren{y}$, $f \paren{z} < z$, $x \perp y$, $x < f \paren{y}$, $x < z$, and $x \perp f \paren{z}$.
        \item\label{SigmaDeterminesSingleOrbits} Suppose $p \in \mathcal{P}$ with $p \subseteq f$. Then there is a $p' \in \mathcal{P}$ such that $p \subseteq p' \subseteq f$, and $p'$ determines $\paren{x,x}$ for all $x \in \dur{p}$.
        \item\label{CanSeparateStrongLessWithFixedPts} (Converse to \Cref{NLemma}) If $x,y \in \mathbf{P}$ with $\orb{f}{x} <_{f}^{s} \orb{f}{y}$, then there exists a fixed point $c = f \paren{c}$ with $x < c < y$.
        \item\label{CanSeparateIncomparableWithFixedPts} (Converse to \Cref{MLemma}) If $\orb{f}{x} \nleq_{f}^{w} \orb{f}{y}$, then there are fixed points $c = f \paren{c}$ and $d = f \paren{d}$ with $y < c$ and $x \nless c$, and $d < x$ and $d \nless y$.
    \end{enumerate}
\end{prop}
\begin{proof}
    \begin{enumerate}
        \item Let $x,y \in \mathbf{P}$. If $\parity{x,f} \neq 0$, then $\mathbf{b}^{f} \paren{x,y}$ is eventually constant on both sides by property \ref{OrbitalIsRandom}. If $\parity{x,f} = 0$ and $\sp{x,f} = n < \infty$, then $\mathbf{b}^{f} \paren{x,y}$ is $n$-periodic on both sides since $f^{n} \paren{x} = x$. If $\sp{x,f} = \infty$ and $y \in f^{\Z} \paren{x}$, then $\mathbf{b}^{f} \paren{x,y}$ is eventually constant $0$ on both sides.
        
        Finally, if $\sp{x,f} = \infty$ and $y \notin f^{\Z} \paren{x}$, then we apply property \ref{CanFitTighteningSpirals} (twice) to obtain $k^{+}, k^{-} \in \Z$, $n^{+}, n^{-} \geq 1$, and $\alpha^{+}, \beta^{+}, \alpha^{-}, \beta^{-} \in \mathbf{P}$ that, by \Cref{TighteningSpirals}, witness $n^{+}$-periodicity of $\mathbf{b}^{f}_{\geq 0} \paren{y, f^{k^{+}} \paren{x}} = \mathbf{b}^{f}_{\geq k^{+}} \paren{y, x}$ and $\mathbf{b}^{f}_{\leq 0} \paren{f^{k^{+}} \paren{x}, y} = \mathbf{b}^{f}_{\leq -k^{+}} \paren{x, x}$, and similarly $n^{-}$-periodicity of $\mathbf{b}^{f}_{\leq k^{-}} \paren{y, x}$ and $\mathbf{b}^{f}_{\geq -k^{+}} \paren{x,y}$.
        
        \item Suppose $\sp{x,f} = \infty$; then $b_{i}^{f} \paren{x,x}$ holds iff $i = 0$. By universality of $\mathbf{P}$, take a subset of $\mathbf{P}$ with the following order type:
        
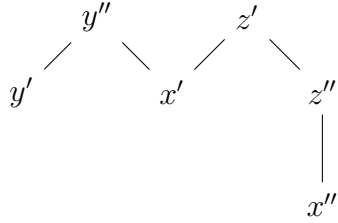
\begin{center}
            \begin{tikzpicture}
                \node (fy) at (-1,1) {$y''$};
                \node (y) at (-2,0) {$y'$};
                \node (x) at (0,0) {$x'$};
                \node (z) at (1,1) {$z'$};
                \node (fz) at (2,0) {$z''$};
                \node (fx) at (2,-1.5) {$x''$};
                \draw (y) -- (fy) -- (x) -- (z) -- (fz) -- (fx);
            \end{tikzpicture}
        \end{center}
        Let $p \in \mathcal{P}$ be defined by $p = \brace{\paren{x',x''}, \paren{y',y''}, \paren{z',z''}}$. Let $g \in \brack{p} \cap D \paren{p}$. By \Cref{MLemma}, $\sp{x',g} = \infty$, and hence $\mathbf{b}^{g} \paren{x',x'} = \mathbf{b}^{f} \paren{x,x}$. By \Cref{DenseClassImpliesUniversality}, let $\phi \from \paren{\dur{p},\mathbf{b}^{g}} \hookrightarrow \mathbf{P}_{f}$ be an embedding of $L$-structures; by ultrahomogeneity of $\mathbf{P}_{f}$, we assume further that $\phi \paren{x'} = x$. Then $y := \phi \paren{y'}$ and $z := \phi \paren{z'}$ are the desired points.

        \item Fix $x \in \dur{p}$.
        \begin{itemize}
            \item Case I: $f^{\Z} \paren{x}$ is an infinite antichain.\\
            Then we may take $y,z \in \mathbf{P}$ as described in part \ref{CanFitMs}. Let $p' := p \cup f \dhr_{\brace{y,z}}$. Then for each $g \in \brack{p'}$, the points $x$, $y$, $z$, $p' \paren{y}$, and $p' \paren{z}$ force $g^{\Z} \paren{x}$ to be an infinite antichain by \Cref{MLemma}, and so $b_{i}^{g} \paren{x,x}$ holds iff $i = 0$.
            \item Case II: $f^{\Z} \paren{x}$ is a finite cycle.\\
            Extend $p$ to $p' := p \cup f\dhr_{f^{\Z} \paren{x}}$. Then any $g \in \brack{p'}$ agrees with $f$ on $f^{\Z} \paren{x}$, and so $\mathbf{b}^{f} \paren{x,x} = \mathbf{b}^{g} \paren{x,x}$.
            \item Case III: $\parity{x,f} \neq 0$.\\
            Suppose $\parity{x,f} = +1$; the proof for negative parity is similar. By property \ref{OrbitalIsRandom}, there is $N$ large enough so that $b_{n}^{f} \paren{x,x}$ holds and $b_{-n}^{f} \paren{x,x}$ fails for all $n \geq N$. Define $p' := p \cup f \dhr_{\brace{f^{i} \paren{x} : \absval{i} \leq N + \sp{x,f}}}$. Then every $g \in \brack{p'}$ has $b_{i}^{g} \paren{x,x} \iff b_{i}^{f} \paren{x,x}$ for all $\absval{i} \leq N + \sp{x,f}$. This implies $\mathbf{b}^{g} \paren{x,x} = \mathbf{b}^{f} \paren{x,x}$ since $b_{i} \paren{x,x} \implies b_{i + \sp{x,f}} \paren{x,x}$ for all $i \in \Z$.
        \end{itemize}
        
        Observe that case I adds two new orbits of non-zero parity, and cases II and III extend $p$ economically. Thus, we may iteratively extend $p$ by listing those points $x \in \dur{p}$ for which $f^{\Z} \paren{x}$ is an infinite antichain, and adding two orbits of nonzero parity for each such $x$ to determine $\paren{x,x}$. Then we economically extend the rest of the orbits by applying cases II and III.

        \item By part \ref{SigmaDeterminesSingleOrbits}, let $p \in \mathcal{P}$ such that $f \dhr_{\brace{x,y}} \subseteq p \subseteq f$, and which determines $\paren{x,x}$ and $\paren{y,y}$. By universal-homogeneity, we can find a point $c' \in \mathbf{P}$ such that $x' < c' < y'$ for all $x' \in \dur{p} \cap \orb{f}{x}$ and $y' \in \dur{p} \cap \orb{f}{y}$. It follows that the map $q := p \cup \brace{\paren{c',c'}}$ is a partial automorphism. For any $g \in \brack{q}$, we have that $c'$ is a fixed point of $g$ with $x < c' < y$, so $\orb{g}{x} <_{g}^{s} \orb{g}{y}$ by \Cref{FixedPointBetweenForcesStronglyLess}. Hence, $\mathbf{b}^{g} \paren{x,y} = \mathbf{1} = \mathbf{b}^{f} \paren{x,y}$ and $\mathbf{b}^{g} \paren{y,x} = \mathbf{0} = \mathbf{b}^{f} \paren{y,x}$. Moreover, since $f$ and $g$ both extend $p$, we have $\mathbf{b}^{g} \paren{x,x} = \mathbf{b}^{f} \paren{x,x}$ and $\mathbf{b}^{g} \paren{y,y} = \mathbf{b}^{f} \paren{y,y}$ by choice of $p$. It follows that $\paren{\brace{x,y}, \mathbf{b}^{g}} \isom \paren{\brace{x,y}, \mathbf{b}^{f}}$. 
        
        Take any $g \in \brack{q} \cap D \paren{q}$. By \Cref{DenseClassImpliesUniversality}, there is an embedding $\phi \from \paren{\brace{x,y,c'}, \mathbf{b}^{g}} \hookrightarrow \mathbf{P}_{f}$, and by ultrahomogeneity we assume $\phi$ fixes $x$ and $y$. Then $c := \phi \paren{c'}$ witnesses the desired result.

        \item The proof is similar to part \ref{CanSeparateStrongLessWithFixedPts}, with the added fixed points in the appropriate locations.\qedhere
    \end{enumerate}
\end{proof}

We claim that $\Sigma$ consists precisely of the generic automorphisms of $\mathbf{P}$ --- that is, $\Sigma = \Gamma$.

\begin{thm}\label{GenericIsInSigma}
    The generic $f \in \Aut{\mathbf{P}}$ is in $\Sigma$.
\end{thm}
\begin{proof}
    \begin{enumerate}[label=(\Alph*)]
        \item[\ref{ClassIsDense}] $\Aut{\mathbf{P}}$ is Polish, so any comeagre conjugacy class is dense.
        
        \item[\ref{Ultrahomogeneity}] Immediate from the case $g = f$ in \Cref{GenericPfIsUltrahomogeneous} below.
        
        \item[\ref{OrbitalIsRandom}] We claim that for every $x \in \mathbf{P}$, the set $$C \paren{x} := \brace{f \in \Aut{\mathbf{P}} : \parity{x,f} \neq 0 \implies \textnormal{$\mathbf{b}^{f}_{\geq 0} \paren{x,x}$ is eventually constant $1$}}$$ has dense interior; it suffices to show that for all $p \in \mathcal{P}$ there is an extension $q \supseteq p$ for which $\brack{q} \subseteq C \paren{x}$.
    
        Let $p \in \mathcal{P}$, and let $f \in \brack{p}$ such that $\parity{x,f} = +1$. Choose $x' \in \orb{f}{x}$ and $N$ large enough so that for all $y \in \dur{p} \cap \orb{f}{x}$, there exist $0 < i \leq j < N$ such that $f^{i} \paren{x'} \leq y \leq f^{j} \paren{x'}$. Without loss of generality, take $N$ a multiple of $\sp{x',f}$. Let $a_{0} := f^{N} \paren{x'}$; by universal-homogeneity of $\mathbf{P}$, let $a_{i} \in \mathbf{P}$ for each $1 \leq i < \sp{x',f}$ be an upper bound of both $x'$ and $f^{N+i-\sp{x',f}} \paren{x}$, such that $a_{i}$ and $f^{N+i-\sp{x',f}} \paren{x}$ have the same types over $\dur{p} \setminus \orb{f}{x}$. Then define $q \in \mathcal{P}$ by: $$q := p \cup f \dhr_{\brace{f^{i} \paren{x'} : 0 \leq i < N}} \cup \brace{\paren{a_{i-1}, a_{i}} : 1 \leq i < \sp{x,f}}.$$
        Then $q \supseteq p$, and for every $g \in \brack{q}$, $\parity{x,g} = +1$ and $x' \leq g^{n} \paren{x'}$ for every $n \geq N$. Hence $\mathbf{b}^{g}_{\geq N} \paren{x',x'}$ is eventually constant $1$; by \Cref{OrbitalIsDirectedEquivs}, this property is orbital-invariant, so $\mathbf{b}^{g}_{\geq N} \paren{x,x}$ is eventually constant $1$. Thus, $\brack{q} \subseteq C \paren{x}$. The same argument applies to the $\parity{x,f} = -1$ case. Thus $C \paren{x}$ has dense interior, so $\bigcap_{x \in \mathbf{P}} C \paren{x}$ is comeagre.
    
        Let $f, g \in \Aut{\mathbf{P}}$. A straightforward calculation shows $\mathbf{b}^{f} \paren{x,x} = \mathbf{b}^{gfg^{-1}} \paren{g \paren{x}, g \paren{x}}$, and $\parity{x,f} = \parity{g \paren{x}, gfg^{-1}}$. Then $f \in C \paren{x}$ iff $gfg^{-1} \in C \paren{g \paren{x}}$. Hence, the intersection $\bigcap_{x \in \mathbf{P}} C \paren{x}$ is conjugacy-invariant, and so must contain every generic element. \qedhere

        \item[\ref{CanFitTighteningSpirals}] For $x \in \mathbf{P}$, $y \subseteq \mathbf{P}$ finite, and $\sigma \in \brace{-1,+1}$, let $S \paren{x,y,\sigma}$ denote the set of $f \in \Aut{\mathbf{P}}$ satisfying property \ref{CanFitTighteningSpirals} for $x$, $y$, and $\sigma$. We claim every such $S \paren{x,y,\sigma}$ has dense interior. Moreover, we assume $\sigma = +1$, since the proof for $\sigma = -1$ is mirrored.

        Let $p \in \mathbf{P}$, and let $f \in \brack{p}$. By extending $p$ to $p \cup f \dhr_{\brace{y}}$, we may assume without loss of generality that $y \in \dur{p}$. Suppose $\sp{x,f} = \infty$ and $y \notin f^{\Z} \paren{x}$ (if no $f \in \brack{p}$ satisfies this then we are done, since then $\brack{p} \subseteq S \paren{x,y,\sigma}$ vacuously).

        Thus, we may repeat part of the proof of Kuske \& Truss that determined maps are cofinal; we outline this argument here for the sake of completeness. By \Cref{InfiniteAntichainTypes}, and since $\dur{p}$ is finite, we may choose $k \in \Z$ and $n > 0$ so that $$\operatorname{qftp}_{<} \paren{f^{k} \paren{x}, \dur{p} \setminus f^{\Z} \paren{x}} = \operatorname{qftp}_{<} \paren{f^{k + n} \paren{x}, \dur{p} \setminus f^{\Z} \paren{x}}.$$
        Moreover, we may choose $n$ large enough so that $f^{k'} \paren{x} \notin \dur{p}$ for all $k' \geq k + n$. Then extend $p$ to agree with $f$ on $\brace{f^{i} \paren{x} : 0 \leq i < k+n}$.

        We then introduce new elements $\alpha_{i}$ and $\beta_{i}$ for $0 \leq i \leq n$, subject to the following:
        \begin{itemize}
            \item $\alpha_{i} \perp \alpha_{j}$ and $\beta_{i} \perp \beta_{j}$ for $0 < \absval{i-j} < n$;
            \item $\alpha_{0} < \alpha_{n}$ and $\beta_{n} < \beta_{0}$;
            \item $\alpha_{i} < f^{k+i} \paren{x} < \beta_{i}$;
            \item $\alpha_{i}$, $\beta_{i}$, and $f^{k+i} \paren{x}$ have the same quantifier-free types over $\dur{p} \setminus f^{\Z} \paren{x}$.
        \end{itemize}

        By the condition on the quantifier-free types, we may realize the $\alpha_{i}$'s and $\beta_{i}$'s as points in $\mathbf{P}$, and the map $q$ given by $$q := p \cup \brace{\paren{\alpha_{i}, \alpha_{i+1}} : 0 \leq i < n} \cup \brace{\paren{\beta_{i}, \beta_{i+1}} : 0 \leq i < n}$$
        is a partial automorphism extending $p$. Moreover, for every $g \in \brack{q}$, the points $\alpha_{0}$ and $\beta_{0}$ witness $g \in S \paren{x,y,\sigma}$.

        Hence every $S \paren{x,y,\sigma}$ is dense open, and so $\bigcap_{x \in \mathbf{P}} \bigcap_{y \in \mathbf{P}} \bigcap_{\absval{\sigma} = 1} S \paren{x,y,\sigma}$ is comeagre.

        Moreover, for any $f,g \in \Aut{\mathbf{P}}$ it is straightforward to show that $f \in S \paren{x,y,\sigma}$ if and only if $gfg^{-1} \in S \paren{g \paren{x}, g \paren{y}, \sigma}$. It follows that $\bigcap_{x \in \mathbf{P}} \bigcap_{y \in \mathbf{P}} \bigcap_{\absval{\sigma} = 1} S \paren{x,y,\sigma}$ is conjugacy-invariant, so it contains every generic element.
    \end{enumerate}
\end{proof}

\begin{thm}\label{GenericPfIsUltrahomogeneous}
    Let $f,g \in \Aut{\mathbf{P}}$ be generic. Then every finite partial $L$-isomorphism $\gamma \from \mathbf{P}_{f} \rightharpoonup \mathbf{P}_{g}$ extends to an isomorphism $\bar{h} \from \mathbf{P}_{f} \to \mathbf{P}_{g}$.
\end{thm}
One may also view this theorem as a strengthening of \Cref{ConjugateIffLIsomorphic}. If $f$ and $g$ are generic, then certainly they are conjugate, and so $\mathbf{P}_{f} \isom \mathbf{P}_{g}$; this theorem says that moreover, a witnessing isomorphism may be taken to extend any desired finite partial $L$-isomorphism.
\begin{proof}
    We show that for every $x \in \mathbf{P} \setminus \dom{\gamma}$, there is $y \in \mathbf{P} \setminus \ran{\gamma}$ for which $\gamma' := \gamma \cup \brace{\paren{x,y}}$ is an $L$-isomorphism. By \Cref{FinitePartialLIsomCanExtend}, we may uniquely extend $\gamma$ to an $L$-isomorphism $h \from f^{\Z} \brack{\dom{\gamma}} \to g^{\Z} \brack{\ran{\gamma}}$. If $x \in \dom{h}$, then clearly $\gamma' := h \dhr_{\dom{\gamma} \cup \brace{x}}$ works, so assume without loss of generality that $x \notin \dom{h}$.
    
    \begin{itemize}
        \item Case I: $x \notin \dom{h}$, $\parity{x,f} = 0$, and $\sp{x,f} = n < \infty$\\
        By $g \in D \paren{g \dhr_{\ran{\gamma}}}$, take $p \subseteq g$ determined such that $\ran{\gamma} \subseteq \operatorname{dur} \paren{p}$. By \Cref{StrongAmalgamationProperty}, define the poset $Q$ to be the result of amalgamating $\dom{h} \cup f^{\Z} \paren{x}$ with $X := g^{\Z} \brack{\dur{p}} \supseteq \ran{h}$ over their common subset $\dom{h} \isom \ran{h}$; let $a_{i}$ denote the image of $f^{i} \paren{x}$ under this amalgamation. We claim $q := p \cup \brace{\paren{a_{0}, a_{1}}, \ldots, \paren{a_{n-1}, a_{0}}}$ is a partial automorphism of $Q$. Indeed, if $a_{i} < y'$ for $y' \in \dom{p}$, then $a_{i} < y'' < y'$ for some $y'' \in \ran{h}$. Then $f^{i} \paren{x} < h^{-1} \paren{y''}$, and so $f^{i+1} \paren{x} < f \paren{h^{-1} \paren{y''}} = h^{-1} \paren{g \paren{y''}}$, and thus $a_{i+1} < g \paren{y''}$. Also, $y'' < y'$ implies $g \paren{y''} < g \paren{y'}$. It follows that $a_{i} < y'$

        \item Case II: $x \notin \dom{h}$ and $\parity{x,f} \neq 0$\\
        We assume $\parity{x,f} = +1$, as the case for $\parity{x,f} = -1$ is symmetric. By $g \in D \paren{g \dhr_{\ran{\gamma}}}$, take $p \subseteq g$ determined such that $\ran{\gamma} \subseteq \operatorname{dur} \paren{p}$, and let $X := g^{\Z} \brack{\dur{p}} \supseteq \ran{h}$.
        
        We let $Q$ be the partial order obtained via \Cref{StrongAmalgamationProperty} by amalgamating $\dom{h} \cup f^{\Z} \paren{x}$ and $X$ over their common subset $\dom{h} \isom \ran{h}$.\footnote{Since $Q$ is infinite, we cannot claim at this point that $Q$ is realized as a subset of $\mathbf{P}$. We will embed a finite fragment of $Q$ into $\mathbf{P}$ later.} Let $a_{i}$ be the image of $f^{i} \paren{x}$ in $Q$.
        
        We extend $Q$ by augmenting it with the following points:
        \begin{itemize}
            \item For each $x' \in \dom{\gamma}$ for which $\orb{f}{x} <_{f}^{s} \orb{f}{x'}$ (respectively $\orb{f}{x'} <_{f}^{s} \orb{f}{x}$), observe that $a_{i} <_{Q} g^{j} \paren{h \paren{x'}}$ (respectively $g^{j} \paren{h \paren{x'}} <_{Q} a_{i}$) for every $i,j \in \Z$. Thus, we add a point $c_{x'}$ with $a_{i} <_{Q} c_{x'} <_{Q} g^{j} \paren{h \paren{x'}}$ (respectively $g^{j} \paren{h \paren{x'}} <_{Q} c_{x'} <_{Q} a_{i}$).
            \item For each $x' \in \dom{\gamma}$ for which $\orb{f}{x} \nleq_{f}^{w} \orb{f}{x'}$, observe that $f^{i} \paren{x} \nleq f^{j} \paren{x'}$, hence $a_{i} \nleq_{Q} g^{j} \paren{h \paren{x'}}$, for every $i,j \in \Z$. Thus, we add a point $d_{x'}$ with $d_{x'} <_{Q} a_{i}$ and $d_{x'} \perp_{Q} g^{j} \paren{h \paren{x'}}$ for all $i,j \in \Z$.
            \item Similarly, for each $x' \in \dom{\gamma}$ for which $\orb{f}{x'} \nleq_{f}^{w} \orb{f}{x}$, we add a point $e_{x'}$ with $a_{i} <_{Q} e_{x'}$ and $e_{x'} \perp_{Q} g^{j} \paren{h \paren{x'}}$ for all $i,j \in \Z$.
        \end{itemize}
        To be precise, we add these points by applying \Cref{StrongAmalgamationProperty}, so these new points are incomparable to all points of $Q$ not mentioned (except those forced to be related by transitivity).
        
        Let $N < \omega$ be large, its value to be determined later. We now claim $$q \paren{z} := \begin{cases}
        p \paren{z} & \text{if $z \in \dom{p}$;}\\
        a_{i+1} & \text{if $z = a_{i}$ for $-N \leq i < N$;}\\
        z & \text{if $z$ is $c_{x'}$, $d_{x'}$, or $e_{x'}$ for some $x'$;}
        \end{cases}$$ is a partial automorphism of $Q$: 
        \begin{itemize}
            \item Note $a_{i} <_{Q} a_{j} \iff f^{i} \paren{x} < f^{j} \paren{x} \iff f^{i+1} \paren{x} < f^{j+1} \paren{x} \iff a_{i+1} <_{Q} a_{j+1} \iff q \paren{a_{i}} <_{Q} q \paren{a_{j}}$ for each $-N \leq i,j < N$. Similarly, $y' <_{Q} y'' \iff y' < y'' \iff g \paren{y'} < g \paren{y''} \iff q \paren{y'} <_{Q} q \paren{y''}$ for each $y',y'' \in \dom{p}$.
            \item Suppose $-N \leq i < N$ and $y' \in \ran{h}$. Then $a_{i} <_{Q} y' \iff f^{i} \paren{x} < h^{-1} \paren{y'} \iff f^{i+1} \paren{x} < f \paren{h^{-1} \paren{y'}} = h^{-1} \paren{g \paren{y'}} \iff a_{i+1} <_{Q} g \paren{y'}$. Then if $y' \in X$, since $\ran{h}$ is $g$-invariant and $Q$ was constructed by amalgamation over $\ran{h}$:
            \begin{align*}
                a_{i} <_{Q} y' &\iff \textnormal{$a_{i} <_{Q} y'' <_{Q} y'$ for some $y'' \in \ran{h}$}\\
                &\iff \textnormal{$a_{i+1} <_{Q} g \paren{y''} <_{Q} g \paren{y'}$ for some $y'' \in \ran{h}$}\\
                &\iff \textnormal{$a_{i+1} <_{Q} y''' <_{Q} g \paren{y'}$ for some $y''' \in \ran{h}$}\\
                &\iff a_{i+1} <_{Q} g \paren{y'}\\
                &\iff q \paren{a_{i}} <_{Q} q \paren{y'}
            \end{align*}
            The same argument also shows $y' <_{Q} a_{i} \iff q \paren{y'} <_{Q} q \paren{a_{i}}$ as well.
            \item By definition of $c_{x'}$, if $x' \in \dom{\gamma}$ with $\orb{f}{x} <_{f}^{s} \orb{f}{x'}$, we have $a_{i} <_{Q} c_{x'}$ for \emph{all} $i$; therefore, $a_{i} <_{Q} c_{x'} \iff a_{i+1} <_{Q} c_{x'} \iff q \paren{a_{i}} <_{Q} q \paren{c_{x'}}$.
            
            Similarly, $c_{x'} <_{Q} g^{j} \paren{h \paren{x'}} \iff c_{x'} <_{Q} g^{j+1} \paren{h \paren{x'}}$ for all $j$. If $c_{x'} <_{Q} y''$ for some $y'' \in X \setminus g^{\Z} \paren{h \paren{x'}}$, then $c_{x'} <_{Q} g^{j} \paren{h \paren{x'}} < y'' \iff c_{x'} <_{Q} g^{j + 1} \paren{h \paren{x'}} < g \paren{y''}$ for all $j \in \Z$; in particular, $c_{x'} <_{Q} y'' \iff q \paren{c_{x'}} <_{Q} g \paren{y''}$ for every $y'' \in X$.
            
            If instead $y'' <_{Q} c_{x'}$, then $y'' <_{Q} a_{i} <_{Q} c_{x'}$ for some $i$, and a similar argument shows $y'' <_{Q} c_{x'} \iff g \paren{y''} <_{Q} q \paren{c_{x'}}$ for each $y'' \in X$.
            
            The same arguments apply if instead $\orb{f}{x'} <_{f}^{s} \orb{f}{x}$. The corresponding statements for $d_{x'}$ and $e_{x'}$ are proven similarly, using the same kind of case-checking.
        \end{itemize}
        
        By the above and by universal-homogeneity of $\mathbf{P}$, there is an embedding $\dur{q} \hookrightarrow \mathbf{P}$ which fixes $\dur{p}$ pointwise; identify the $a_{i}$'s, the $c_{x'}$'s, and the $d_{x'}$'s with their image under this embedding. Then $q \in \mathcal{P}$ with $p \subseteq q$; by cofinality, take $p'$ determined with $q \subseteq p'$. Then $g \in E \paren{p,p'}$ gives $r \in \mathcal{P}$ with $\dom{r} = \dur{p'}$, $r$ fixes $\dur{p}$ pointwise, and $r \circ p' \circ r^{-1} \subseteq g$. We claim $\gamma' := h \cup \brace{\paren{x, y}}$ is a finite partial $L$-isomorphism, where $y := r \paren{a_{0}}$. At this point, we decide how big to make $N$. Let $n := \sp{x,f}$. Then by \Cref{OrbitalIsDirectedEquivs} we choose $N$ to be large enough so that whenever $x' \in \dom{\gamma}$, we have:
        \begin{itemize}
            \item $\mathbf{b}_{\geq N - n}^{f} \paren{x,x} = \mathbf{1}$, and $b_{\leq -N + n}^{f} \paren{x,x} = \mathbf{0}$;
            \item $\mathbf{b}_{\geq N - n}^{f} \paren{x,x'} = \mathbf{1}$ if $\orb{f}{x} \leq_{f}^{w} \orb{f}{x'}$;
            \item $\mathbf{b}_{\geq N - n}^{f} \paren{x',x} = \mathbf{1}$ if $\orb{f}{x'} \leq_{f}^{w} \orb{f}{x}$;
            \item $\mathbf{b}_{\leq -N + n}^{f} \paren{x,x'} = \mathbf{0}$ if $\orb{f}{x} \nless_{f}^{s} \orb{f}{x'}$;
            \item $\mathbf{b}_{\leq -N + n}^{f} \paren{x',x} = \mathbf{0}$ if $\orb{f}{x'} \nless_{f}^{s} \orb{f}{x}$.
        \end{itemize}
        Then:
        \begin{itemize}
            \item For $-N \leq i \leq N$:
            \begin{align*}
                x < f^{i} \paren{x} &\iff a_{0} < a_{i}\\
                &\iff r \paren{a_{0}} < r \paren{a_{i}} = g^{i} \paren{r \paren{a_{0}}}
            \end{align*}
            Hence $b_{i}^{f} \paren{x,x} \iff b_{i}^{g} \paren{y,y}$ for all $\absval{i} \leq N$. By choice of $N$, this forces $\mathbf{b}^{f} \paren{x,x} = \mathbf{b}^{g} \paren{y,y}$.
            \item For $x' \in \dom{\gamma}$, we have:
            \begin{align*}
                x' < f^{i} \paren{x} &\iff \gamma \paren{x'} = h \paren{x'} < a_{i}\\
                &\iff \gamma \paren{x'} = r \paren{\gamma \paren{x'}} < r \paren{a_{i}} = g^{i} \paren{r \paren{a_{0}}}
            \end{align*}
            Hence $b_{i}^{f} \paren{x',x} = b_{i}^{g} \paren{\gamma \paren{x'},y}$ for all $\absval{i} \leq N$. By choice of $N$, this forces $\mathbf{b}^{f} \paren{x',x} = \mathbf{b}^{g} \paren{\gamma \paren{x'},y}$.
            \item For $x' \in \dom{\gamma}$, we have:
            \begin{align*}
                x < f^{i} \paren{x'} &\iff f^{-i} \paren{x} < x'\\
                &\iff a_{-i} < h \paren{x'}\\
                &\iff g^{-i} \paren{r \paren{a_{0}}} = r \paren{a_{-i}} < r \paren{h \paren{x'}} = \gamma \paren{x'}\\
                &\iff r \paren{a_{0}} < g^{i} \paren{\gamma \paren{x'}}
            \end{align*}
            Hence $b_{i}^{f} \paren{x,x'} = b_{i}^{g} \paren{y,\gamma \paren{x'}}$ for all $\absval{i} \leq N$. By choice of $N$, this forces $\mathbf{b}^{f} \paren{x,x'} = \mathbf{b}^{g} \paren{y,\gamma \paren{x'}}$.
        \end{itemize}
        \item Case III: $x \notin \dom{h}$ and $\sp{x,f} = \infty$\\
        Our technique is the same as for Case III, but the $q$ we construct is different. Let $p$ and $X$ be as in Case III. We first obtain $Q$ by amalgamating $\dom{h} \cup f^{\Z} \paren{x}$ and $X$ over $\dom{h}$ as before, then augmenting $Q$ with the same $c_{x'}$'s, $d_{x'}$'s and $e_{x'}$'s, but we also add points $\sigma_{0}$, $\sigma_{1}$, $\tau_{0}$, and $\tau_{1}$ arranged in an ``M''-configuration around $a_{0}$ (as shown in \Cref{fig:DetermineInfiniteSpiral} later).
            
        In \Cref{GenericIsInSigma}, we have already shown that the generic $f$ satisfies properties \ref{OrbitalIsRandom} and \ref{CanFitTighteningSpirals}. Moreover, recall that these two properties are the only ones needed to prove \Cref{PropertiesOfSigma}\ref{AllSequencesEventuallyPeriodic}. Thus, for all $x' \in h^{-1} \brack{\dur{p}}$, there are $k \in \Z$ and $n \geq 1$ such that $\mathbf{b}^{f}_{\geq k} \paren{x,x'}$ is $n$-periodic, and similarly for $\mathbf{b}^{f}_{\geq k} \paren{x',x}$. Note that replacing $k$ with any $k' \geq k$ preserves periodicity. Also, an analogous statement is true for the left-hand tails as well. Thus, by taking $k$ large enough, and taking the least common multiple of all the resulting $n$'s, we may assume that $k \geq 0$ and $n \geq 1$ are such that $\mathbf{b}^{f}_{\geq k} \paren{x,x'}$, $\mathbf{b}^{f}_{\geq k} \paren{x',x}$, $\mathbf{b}^{f}_{\leq -k} \paren{x,x'}$, and $\mathbf{b}^{f}_{\leq -k} \paren{x',x}$ are all $n$-periodic for all $x' \in h^{-1} \brack{\dur{p}}$.
            
        In particular, note that $f^{k} \paren{x}$ and $f^{k+n} \paren{x}$ have the same order-types over the set $h^{-1} \brack{\dur{p}}$, and so $a_{k}$ and $a_{k+n}$ have the same order types over $\ran{h} \cap \dur{p}$. Likewise, $a_{-k}$ and $a_{-k-n}$ have the same order types over $\ran{h} \cap \dur{p}$. We now add tightening spirals for the sequence of $a_{i}$'s, as follows. Take points $\alpha_{i}^{+}, \beta_{i}^{+}, \alpha_{-i}^{-}, \beta_{-i}^{-}$ for $0 \leq i \leq n$ such that:
        \begin{itemize}
            \item $\alpha_{i}^{+} \perp \alpha_{j}^{+}$ whenever $0 < \absval{i-j} < n$, and similarly for the $\beta_{i}^{+}$'s, $\alpha_{i}^{-}$'s, and $\beta_{i}^{-}$'s;
            \item $\alpha_{0}^{+} < \alpha_{n}^{+}$, $\alpha_{0}^{-} < \alpha_{n}^{-}$, $\beta_{n}^{+} < \beta_{0}^{+}$, and $\beta_{n}^{-} < \beta_{0}^{-}$;
            \item $\alpha_{i}^{+} < a_{k+i} < \beta_{i}^{+}$ and $\alpha_{i}^{-} < a_{-k-i} < \beta_{i}^{-}$;
            \item $\alpha_{i}^{+} \perp \alpha_{j}^{-}$ and $\beta_{i}^{+} \perp \beta_{j}^{-}$ for all $i,j \leq n$;
            \item $\alpha_{i}^{+} \perp \beta_{j}^{-}$ and $\alpha_{i}^{-} \perp \beta_{j}^{+}$ for all $i,j \leq n$, with the possible exception of $\alpha_{0}^{+} < \beta_{0}^{-}$ and $\alpha_{0}^{-} \perp \beta_{0}^{+}$ in case $k=0$;
            \item For each $0 \leq i \leq n$, $\alpha_{i}^{+}$, $\beta_{i}^{+}$, and $a_{k+i}$ have the same types over $\ran{h} \cap \dur{p}$;
            \item For each $0 \leq i \leq n$, $\alpha_{i}^{-}$, $\beta_{i}^{-}$, and $a_{-k-i}$ have the same types over $\ran{h} \cap \dur{p}$.
        \end{itemize}
        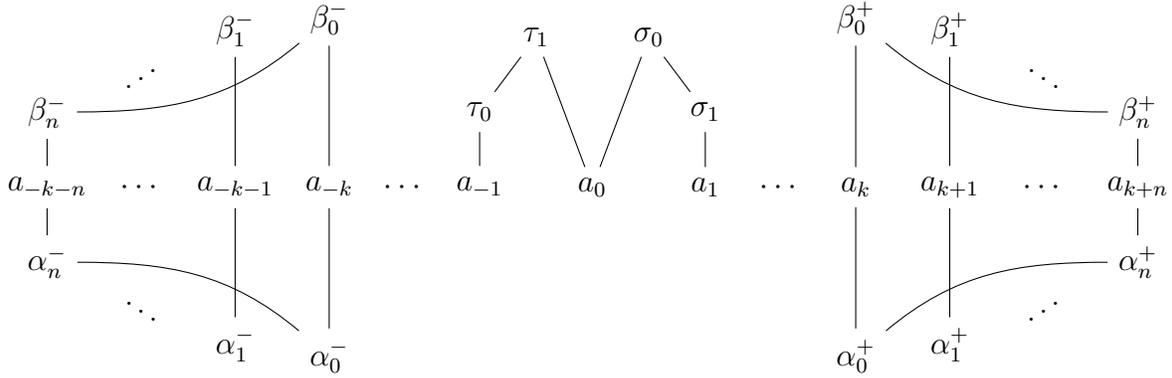
\begin{figure}[h]
            \centering
            \begin{tikzpicture}
                \node (t1) at (-0.75,2) {$\tau_{1}$};
                \node (t0) at (-1.5,1) {$\tau_{0}$};
                \node (s0) at (0.75,2) {$\sigma_{0}$};
                \node (s1) at (1.5,1) {$\sigma_{1}$};
                \node (a0) at (0,0) {$a_{0}$};
                \node (a-1) at (-1.5,0) {$a_{-1}$};
                \node (a1) at (1.5,0) {$a_{1}$};
                \node (a-2) at (-2.5,0) {$\cdots$};
                \node (a2) at (2.5,0) {$\cdots$};
                \node (a-3) at (-3.5,0) {$a_{-k}$};
                \node (a3) at (3.5,0) {$a_{k}$};
                \node (a-4) at (-4.75, 0) {$a_{-k-1}$};
                \node (a4) at (4.75, 0) {$a_{k+1}$};
                \node (a-5) at (-6, 0) {$\cdots$};
                \node (a5) at (6, 0) {$\cdots$};
                \node (a-6) at (-7.25, 0) {$a_{-k-n}$};
                \node (a6) at (7.25, 0) {$a_{k+n}$};
                \node (a0+) at (3.5, -2.25) {$\alpha_{0}^{+}$};
                \node (b0+) at (3.5, 2.25) {$\beta_{0}^{+}$};
                \node (a0-) at (-3.5, -2.25) {$\alpha_{0}^{-}$};
                \node (b0-) at (-3.5, 2.25) {$\beta_{0}^{-}$};
                \node (a1+) at (4.75, -2.1) {$\alpha_{1}^{+}$};
                \node (b1+) at (4.75, 2.1) {$\beta_{1}^{+}$};
                \node (a1-) at (-4.75, -2.1) {$\alpha_{1}^{-}$};
                \node (b1-) at (-4.75, 2.1) {$\beta_{1}^{-}$};
                \node (a2+) at (6, -1.55) {\reflectbox{$\ddots$}};
                \node (b2+) at (6, 1.55) {$\ddots$};
                \node (a2-) at (-6, -1.55) {$\ddots$};
                \node (b2-) at (-6, 1.55) {\reflectbox{$\ddots$}};
                \node (a3+) at (7.25, -1) {$\alpha_{n}^{+}$};
                \node (b3+) at (7.25, 1) {$\beta_{n}^{+}$};
                \node (a3-) at (-7.25, -1) {$\alpha_{n}^{-}$};
                \node (b3-) at (-7.25, 1) {$\beta_{n}^{-}$};
                \draw (a-1) -- (t0) -- (t1) -- (a0) -- (s0) -- (s1) -- (a1);
                \draw (a0-) to[out=140,in=0] (a3-);
                \draw (a0+) to[out=40,in=180] (a3+);
                \draw (b0-) to[out=220,in=0] (b3-);
                \draw (b0+) to[out=-40,in=180] (b3+);
                \draw (a0+) -- (a3) -- (b0+);
                \draw (a1+) -- (a4) -- (b1+);
                \draw (a3+) -- (a6) -- (b3+);
                \draw (a0-) -- (a-3) -- (b0-);
                \draw (a1-) -- (a-4) -- (b1-);
                \draw (a3-) -- (a-6) -- (b3-);
            \end{tikzpicture}
            \caption{Blueprint for the future home of $\gamma' \paren{x}$, and some of the data needed to determine its $\mathbf{b}^{g}$-sequences.}
            \label{fig:DetermineInfiniteSpiral}
        \end{figure}

        We would like to amalgamate $Q$ with the points shown in \Cref{fig:DetermineInfiniteSpiral}, but we cannot guarantee that (for example) $a_{k}$ and $a_{k+n}$ have the same type over $\ran{h}$, so adding the spirals may create new relations between points of $\ran{h}$. We solve this problem by now only considering a finite fragment of $Q$. We obtain the finite poset $Q'$ by amalgamating the sets $$\dur{p} \cup \brace{a_{i} : \absval{i} \leq k+n} \cup \brace{\sigma_{i},\tau_{i} : i=0,1} \cup \brace{c^{x'},d^{x'},e^{x'} : x' \in \dom{\gamma}}$$ and $$\brace{a_{i} : \absval{i} \leq k+n} \cup \paren{\ran{h} \cap \dur{p}} \cup \brace{\alpha_{i}^{+},\alpha_{i}^{-},\beta_{i}^{+},\beta_{i}^{-} : 0 \leq i \leq n}$$ over the common subset $\brace{a_{i} : \absval{i} \leq k+n} \cup \paren{\ran{h} \cap \dur{p}}$.
        
        We now claim $q$ is a partial automorphism on $Q'$, where:
        $$q \paren{z} := \begin{cases}
        p \paren{z} & \text{if $z \in \dom{p}$;}\\
        a_{i+1} & \text{if $z = a_{i}$ for $-k-n \leq i < k+n$;}\\
        z & \text{if $z = c^{x'},d^{x'},e^{x'}$;}\\
        \sigma_{1} & \text{if $z = \sigma_{0}$;}\\
        \tau_{1} & \text{if $z = \tau_{0}$;}\\
        \alpha_{i+1}^{+} & \text{if $z = \alpha_{i}^{+}$ for some $0 \leq i < n$;}\\
        \beta_{i+1}^{+} & \text{if $z = \beta_{i}^{+}$ for some $0 \leq i < n$;}\\
        \alpha_{i-1}^{-} & \text{if $z = \alpha_{i}^{-}$ for some $0 < i \leq n$;}\\
        \beta_{i-1}^{-} & \text{if $z = \beta_{i}^{-}$ for some $0 < i \leq n$.}
        \end{cases}$$
        Most of the cases to check are straightforward; the main case we must be careful of are the relations between the added tightening spirals and the points of $\dur{p}$. Let $0 \leq i < n$ and $y' \in \dom{p}$, such that $\alpha_{i}^{+} < y'$. By the amalgamation that defined $Q'$, it must be that either $\alpha_{i}^{+} < a_{j} < y'$ for some $\absval{j} \leq k+n$, or else $\alpha_{i}^{+} < y'' < y'$ for some $y'' \in \dur{p} \cap \ran{h}$.

        In the case where $\alpha_{i}^{+} < y'' < y'$ for some $y'' \in \dur{p} \cap \ran{h}$, then by the restriction on order-types, we must have $a_{k+i} < y''$. Then $a_{k+i} < y'' \iff f^{k+i} \paren{x} < h^{-1} \paren{y''} \iff f^{k+i+1} \paren{x} < f \paren{h^{-1} \paren{y''}} = h^{-1} \paren{g \paren{y''}} \iff a_{k+i+1} < g \paren{y''}$, and $y'' < y' \iff g \paren{y''} < g \paren{y'} = p \paren{y'}$. Thus, $\alpha_{i+1}^{+} < a_{k+i+1} < p \paren{y'}$.
        
        In the case where $\alpha_{i}^{+} < a_{j} < y'$, we must have either $j = k+i$, or if $i=0$, possibly $j = k+n$. If $j = k+i < k+n$, then $\alpha_{i+1}^{+} < a_{k+i+1}$. By the amalgamation that defined $Q$, we have $a_{k+i} <_{Q} y'$ implies $a_{k+i} < y'' < y'$ for some $y'' \in \ran{h}$. By the previous case, this implies $\alpha_{i}^{+} < a_{k+i+1} < p \paren{y'}$. Otherwise, if $i = 0$ and $j = k+n$, then $a_{k+n} < y'$ implies there is $y'' \in \dur{p} \cap \ran{h}$ with $a_{k+n} < y'' < y'$. By the choice of $k$ and $n$, we must have $a_{k} < y''$ also, and so $\alpha_{0}^{+} < a_{k} < y'' < y'$. Then $\alpha_{1}^{+} < p \paren{y'}$ by the $j=k+i$ case.

        Essentially the same arguments show that $\alpha_{i+1}^{+} < p \paren{y'}$ implies $\alpha_{i}^{+} < y'$, so that $\alpha_{i}^{+} < y' \iff \alpha_{i+1}^{+} < p \paren{y'}$. One similarly shows that $y' < \alpha_{i}^{+} \iff p \paren{y'} < \alpha_{i+1}^{+}$, and also the corresponding statements for $\alpha_{i}^{-}$, $\beta_{i}^{+}$, and $\beta_{i}^{-}$. Thus $q$ is a partial automorphism.
        
        By universal-homogeneity, embed $\dur{q} \hookrightarrow \mathbf{P}$ fixing $\dur{p}$ pointwise. Extend $q$ to $p' \in \mathcal{P}$ determined, so that $g \in E \paren{p,p'}$ implies there exists $r \from \dur{p'} \hookrightarrow \mathbf{P}$, fixing $\dur{p}$ pointwise and such that $r \circ p' \circ r^{-1} \subseteq g$. Let $y := r \paren{a_{0}}$ and $\gamma' := \gamma \cup \brace{\paren{x,y}}$. Then:
        \begin{itemize}
            \item $y$, $r \paren{\sigma_{i}}$, $r \paren{\tau_{i}}$ for $i=0,1$ are arranged in an M, so $\sp{y,g} = \infty$ by \Cref{MLemma}. Thus $b_{i}^{g} \paren{y,y}$ holds iff $i = 0$, and so $\mathbf{b}^{g} \paren{y,y} = \mathbf{b}^{f} \paren{x,x}$.
            \item By \Cref{TighteningSpirals}, the $r \paren{\alpha_{i}^{+}}$'s, $r \paren{\beta_{i}^{+}}$'s, $r \paren{\alpha_{i}^{-}}$'s, and $r \paren{\beta_{i}^{-}}$'s are tightening spirals that witness that $\mathbf{b}^{g}_{\geq k} \paren{y, \gamma \paren{x'}}$ is $n$-periodic, and must agree with $\mathbf{b}^{f}_{\geq k} \paren{x,x'}$ on the full period; similarly for $\mathbf{b}^{g}_{\geq k} \paren{\gamma \paren{x'}, y}$, $\mathbf{b}^{g}_{\leq -k} \paren{y, \gamma \paren{x'}}$, and $\mathbf{b}^{g}_{\leq -k} \paren{\gamma \paren{x'}, y}$. Thus $\mathbf{b}^{g} \paren{y, \gamma \paren{x'}} = \mathbf{b}^{f} \paren{x',x}$ and $\mathbf{b}^{g} \paren{\gamma \paren{x'},y} = \mathbf{b}^{f} \paren{x',x}$ as desired.
        \end{itemize}
    \end{itemize}
    In all cases, $\gamma'$ is a partial $L$-isomorphism with $x \in \dom{\gamma'}$ as desired. This was ``forth'', and ``back'' is the same argument since the hypotheses on $f$ and $g$ are the same.
\end{proof}

\begin{thm}\label{SigmaIsGeneric}
    $\Sigma \subseteq \Gamma$. That is, each element of $\Sigma$ is generic.
\end{thm}
\begin{proof}
    Fix $p,p' \in \mathcal{P}$ determined with $p \subseteq p'$, and suppose $f \in \brack{p} \cap \Sigma$. Let $g \in \brack{p'}$; since $p$ is determined, and $f,g$ both extend $p$, we have $\id_{\dur{p}} \from \paren{\dur{p}, \mathbf{b}^{g}} \to \paren{\dur{p}, \mathbf{b}^{f}}$ is an isomorphism. Since $p'$ is determined, we have $g \in \brack{p'} \subseteq D \paren{p'}$, so by \Cref{DenseClassImpliesUniversality}, there is an embedding $r \from \paren{\dur{p'}, \mathbf{b}^{g}} \hookrightarrow \mathbf{P}_{f}$. By ultrahomogeneity of $\mathbf{P}_{f}$ (property \ref{Ultrahomogeneity}), we may assume $r$ is the identity on the common sub-$L$-structure $\dur{p}$. Then $r$ witnesses $f \in E \paren{p,p'}$.

    Let $p \in \mathcal{P}$; we show $f \in D \paren{p}$. First, we may assume without loss of generality (by economically extending) that $p$ agrees with $f$ on the set $$\brace{f^{i} \paren{x} \in \mathbf{P} : \textnormal{$\exists k, \ell \in \Z$ such that $k \leq i \leq \ell$ and $f^{k} \paren{x}, f^{\ell} \paren{x} \in \dur{p}$}}.$$
    This is so that $f$ ``sees'' that separate $p$-orbits really are separate. Consequently, if $y \in \dur{p} \cap f^{\Z} \paren{x}$, then $y = p^{k} \paren{x}$ for some $k$, i.e. $p$ is defined on the intermediate steps. (This extension is a one-time step, and all further extensions we make will preserve this notationally-convenient property.)
    
    Let $x,y \in \dur{p}$; we will extend $p$ to $q$ that determines $\paren{x,y}$ and $\paren{y,x}$.
    \begin{itemize}
        \item Case I: $y \in f^{\Z} \paren{x}$.\\
        By \Cref{PropertiesOfSigma}\ref{SigmaDeterminesSingleOrbits}, take $q \in \mathcal{P}$ so that $p \subseteq q \subseteq f$ and $q$ determines $\paren{x,x}$. Then $q$ determines $\paren{x,y}$ and $\paren{y,x}$ by \Cref{IsoOverPairOrbitInvariant}. Recall that we can take $q$ to be an economic extension of $p$ if $\sp{x,f} \neq \emptyset$, and $q$ needs only add two orbits of non-zero parity if $\sp{x,f} = \infty$.

        \item Case II: One of the orbits (say $f^{\Z} \paren{x}$) is an infinite antichain.\\
        By property \ref{CanFitTighteningSpirals}, there are $k^{+}, k^{-} \in \Z$ and pairs of points $\alpha^{+}, \beta^{+}$ and $\alpha^{-}, \beta^{-} \in \mathbf{P}$ that form tightening spirals around $f^{k^{+}} \paren{x}$ and $f^{k^{-}} \paren{x}$ respectively, with lengths $n^{+} \geq 1$ and $n^{-} \geq 1$ respectively. Let $q$ be the result of extending $p$ to agree with $f$ on the points $f^{i} \paren{x}$ for $\absval{i} \leq \absval{k^{+}} + \absval{k^{-}} + n^{+} + n^{-}$, the points $f^{i} \paren{\alpha^{+}}$ and $f^{i} \paren{\beta^{+}}$ for $0 \leq i < n$, and the points $f^{i} \paren{\alpha^{-}}$ and $f^{i} \paren{\beta^{-}}$ for $-n \leq i < 0$. Then for every $g \in \brack{q}$, the $\alpha$'s and $\beta$'s witness that $\mathbf{b}^{g}_{\geq -k^{-}} \paren{x,y}$, $\mathbf{b}^{g}_{\geq k^{+}} \paren{y,x}$, $\mathbf{b}^{g}_{\leq -k^{+}} \paren{x,y}$, and $\mathbf{b}^{g}_{\leq k^{-}} \paren{y,x}$ are all periodic (with either period $n^{+}$ or $n^{-}$ accordingly) and agree with the $\mathbf{b}^{f}$'s. Thus $q$ determines $\paren{x,y}$ and $\paren{y,x}$.
        
        \item Case III: One of the orbits (say $f^{\Z} \paren{y}$) has non-zero parity.\\
        Suppose $\parity{y,f} = +1$, as the negative parity case is similar.
        We separate into subcases:
        \begin{enumerate}[label=(\alph*)]
            \item $\orb{f}{x} <_{f}^{s} \orb{f}{y}$:\\
            By \Cref{PropertiesOfSigma}\ref{CanSeparateStrongLessWithFixedPts}, let $c$ be a fixed point with $x < c < y$ and set $q := p \cup \brace{\paren{c,c}}$. Then \Cref{FixedPointBetweenForcesStronglyLess} implies $\mathbf{b}^{f} \paren{x,y} = \mathbf{b}^{g} \paren{x,y} = \mathbf{1}$ and $\mathbf{b}^{f} \paren{y,x} = \mathbf{b}^{g} \paren{y,x} = \mathbf{0}$.
            
            \item $\orb{f}{x} \leq_{f}^{w} \orb{f}{y}$ but $\orb{f}{x} \nless_{f}^{s} \orb{f}{y}$:\\
            By \ref{OrbitalIsRandom}, $\mathbf{b}^{f} \paren{x,y}$ is eventually constant $1$ on the right, and eventually constant $0$ on the left. By \Cref{OrbitalIsDirectedEquivs}, take $N$ large enough so that $n \geq N$ implies $y \leq f^{n} \paren{y}$. Choose $M$ large enough so that $i \geq M$ implies $x \leq f^{i} \paren{y}$ and $x \nleq f^{-i} \paren{y}$. Then let $q := p \cup f \dhr_{\brace{f^{i} \paren{y} : \absval{i} \leq N + M + \sp{y,f}}}$. Let $g \in \brack{q}$; note that $b_{i}^{f} \paren{x,y} = b_{i}^{g} \paren{x,y}$ for all $\absval{i} \leq M + N$.
            
            Since $f^{i} \paren{y} = g^{i} \paren{y}$ for $0 \leq i \leq N + \sp{y,f}$, it follows that $\sp{y,f} = \sp{y,g}$ and $y \leq g^{n} \paren{y}$ for all $n \geq N$. Hence $n \geq N$ implies $x \leq f^{M} \paren{y} = g^{M} \paren{y} \leq g^{M+n} \paren{y}$, so $b_{M+n}^{g} \paren{x,y}$ holds. Thus, $\mathbf{b}_{\geq M+N}^{g} \paren{x,y} = \mathbf{b}_{\geq M+N}^{f} \paren{x,y} = \mathbf{1}$.
            
            Suppose for contradiction $b_{j}^{g} \paren{x,y}$ holds for some $j \leq -M-N$. Since $N \geq \sp{y,f} = \sp{y,g}$, there is $k \geq 0$ such that $-M-N < j + k\cdot \sp{y,g} \leq -M$. Then $x \leq g^{j} \paren{y} \leq g^{j + k \sp{y,g}} \paren{y} = f^{j + k \sp{y,f}} \paren{y}$, so $b_{j + k \sp{y,f}}^{f} \paren{x,y}$ holds, contradicting our choice of $M$. Hence $\mathbf{b}_{\leq -M-N}^{g} \paren{x,y} = \mathbf{b}_{\leq -M-N}^{f} \paren{x,y} = \mathbf{0}$, and so $\mathbf{b}^{g} \paren{x,y} = \mathbf{b}^{f} \paren{x,y}$, as desired.
            
            \item $\orb{f}{x} \nleq_{f}^{w} \orb{f}{y}$:\\
            By \Cref{PropertiesOfSigma}\ref{CanSeparateIncomparableWithFixedPts}, extend to $q := p \cup \brace{\paren{c,c}}$, where $c$ is a fixed point of $f$ with $y < c$ and $x \nless c$. Then \Cref{NLemma} implies $\mathbf{b}^{g} \paren{x,y} = \mathbf{0} = \mathbf{b}^{f} \paren{x,y}$ for all $g \in \brack{q}$.
        \end{enumerate}
        The above subcases describe how to extend $p$ to determine $\paren{x,y}$; a similar, appropriately symmetric consideration of subcases shows how to determine $\paren{y,x}$.
        
        \item Case IV: One of the orbits (say $f^{\Z} \paren{y}$) is a finite cycle.\\
        Let $q := p \cup f \dhr_{f^{\Z} \paren{y}}$. Then for every $g \in \brack{q}$ we have $g^{k} \paren{y} = f^{k} \paren{y}$ for every $k \in \Z$, so $\mathbf{b}^{g} \paren{x,y} = \mathbf{b}^{f} \paren{x,y}$ and $\mathbf{b}^{g} \paren{y,x} = \mathbf{b}^{f} \paren{y,x}$.
        
    \end{itemize}
    We must justify why this procedure terminates. Observe that Cases I and II only add orbits of non-zero parity and Case III can only add fixed points. Thus, we may list the points $x \in \dur{p}$ for which $\sp{x,f} = \infty$, and add finitely many orbits of non-zero parity to determine $\paren{x,y}$ and $\paren{y,x}$ for each such $x$ and each $y \in \dur{p}$ --- two orbits if $y = x$ by \Cref{PropertiesOfSigma}\ref{CanFitMs}, or four if $y \neq x$ by Case II above. Let $q \supseteq p$ be the result of these additions. Then, use Case III to determine $\paren{x,y}$ and $\paren{y,x}$ for all $x \in \dur{q}$ not previously listed and all $y \in \dur{q}$. Note that Case III, unlike Case II, can only extend economically or add fixed points. Also note that $\paren{x,c}$ and $\paren{c,x}$ are automatically determined if $c$ is a fixed point. Finally, apply \Cref{PropertiesOfSigma}\ref{SigmaDeterminesSingleOrbits} to economically extend $q$ to $q'$ which determines $\paren{y,y}$ for all $\sp{y,f} < \infty$. Then all pairs of $q'$-orbits are determined, hence $q'$ is determined. Moreover, $q' \subseteq f$, so $q'$ witnesses $f \in D \paren{p}$.
\end{proof}

This concludes our characterization:
\begin{cor}
    $f \in \Aut{\mathbf{P}}$ is generic if and only if $f \in \Sigma$.
\end{cor}
\begin{proof}
    $\Gamma \subseteq \Sigma$ by \Cref{GenericIsInSigma}, and $\Sigma \subseteq \Gamma$ by \Cref{SigmaIsGeneric}.
\end{proof}

\subsection{Structure of the orbital quotient}\label{subsec:StructureOfTheOrbitalQuotient}

In \cite{Tru92}, Truss gave the following characterization of the generic automorphism of $\Q$. In our terminology, $f \in \Aut{\Q}$ is generic iff $\OQ{f}{\Q} \isom \Q$, and for each $\sigma \in \brace{-1,0,1}$, the set $\brace{\orb{f}{x} : \parity{x,f} = \sigma}$ is dense in $\OQ{f}{\Q}$.

One may ask if an analogous statement for $\mathbf{P}$ is true, keeping in mind that such a statement for $\mathbf{P}$ must also handle the new types of behavior not present in the linear order case --- namely, parity $0$ orbitals (finite and infinite antichains), and the distinction between $<^{w}$ and $<^{s}$.

Indeed, we are able to prove one direction of a suitable analogue. More precisely:
\begin{thm}\label{GenericOrbitalQuotientIsRandom}
    For the generic $f \in \Aut{\mathbf{P}}$, the strong order $<_{f}^{s}$ on $\OQ{f}{\mathbf{P}}$ is isomorphic to $\mathbf{P}$. Moreover, for every $\sigma \in \brace{-1,1}$ and $1 \leq n \leq \infty$, the sets $\brace{\orb{f}{x} : \parity{x,f} = \sigma}$ and $\brace{\orb{f}{x} : \parity{x,f} = 0 \wedge \sp{x,f} = n}$ are dense\footnote{When referring to subsets of a poset, we only ever mean dense \textit{in intervals}; we will never use the notion of dense sets which commonly appears in treatments of forcing.} in $\OQ{f}{\mathbf{P}}$.
\end{thm}

This theorem is achieved with a standard back-and-forth argument. For the ``forth'' step, by the sandwich principle (\Cref{SandwichPrinciple}) it suffices to show that between any finite sets of desired immediate predecessors and immediate successors, there exists an orbital to which we can map the new point.
\begin{lem}[Orbital sandwich lemma]\label{Sandwich}
    Let $A$, $B$, and $C$ be finite subsets of $\mathbf{P}$, and let $f \in \Gamma$, satisfying the following:
    \begin{itemize}
        \item $x \not\sim_{f} x'$ and $y \not\sim_{f} y'$ whenever $x,x' \in A$ are distinct and $y,y' \in B$ are distinct;
        \item $\orb{f}{x} <_{f}^{s} \orb{f}{y}$ for each $x \in A$ and $y \in B$;
        \item $\orb{f}{z} \nleq_{f}^{w} \orb{f}{x}$ and $\orb{f}{y} \nleq_{f}^{w} \orb{f}{z}$ for each $x \in A$, $y \in B$, and $z \in C$.
    \end{itemize}
    Also, let $Q$ be a finite poset and $p_{Q}$ a finite partial automorphism of $Q$ such that $Q = \dur{p_{Q}}$. Then there is an isomorphic copy $Q' \subseteq \mathbf{P}$ of $Q$, and a corresponding $p_{Q'} \in \mathcal{P}$, such that:
    \begin{itemize}
        \item $p_{Q'} \subseteq f$;
        \item For all $x \in A$, $y \in B$, $z \in C$, and $w \in Q'$, we have $\orb{f}{x} <_{f}^{s} \orb{f}{w} <_{f}^{s} \orb{f}{y}$ and $\orb{f}{z} \perp_{f}^{s} \orb{f}{w}$.
    \end{itemize}
\end{lem}

Before giving the proof, we outline the idea. We wish to force a copy of $Q$ to lie in $\mathbf{P}$, where the action of $f$ on that copy agrees with the action of $p_{Q}$. Treating $Q$ as a ``point'', $A$ will be its desired immediate predecessors, $B$ will be the desired immediate successors, and $C$ will be points we wish to force to be incomparable to $Q$. We then apply the sandwich principle: $A$ and $B$ are the bread, the copy of $Q$ is the filling, and $C$ is the side of chips that we wish to keep separate from the sandwich. We accomplish this by adding fixed points between $A$, $Q$, and $B$: these are the toothpicks holding the sandwich together.

\begin{proof}
    Since $f \in \Gamma \subseteq D \paren{f\dhr_{A \cup B \cup C}}$, we let $p \in \mathcal{P}$ be determined, such that $p \subseteq f$ and $A \cup B \cup C \subseteq \dur{p}$. We define a new finite poset $R$ as follows:
    \begin{itemize}
        \item $R := \dur{p} \sqcup \brace{a_{x} : x \in A} \sqcup \brace{b_{y} : y \in B} \sqcup Q$;
        \item Elements of $\dur{p}$ are related in the same way as they were in $\mathbf{P}$, and elements of $Q$ are related the same way as they were in $Q$;
        \item $x <_{R} a_{x} <_{R} w <_{R} b_{y} <_{R} y$ for each $x \in A$, $y \in B$, and $w \in Q$.
    \end{itemize}
    Those elements of $\dur{p}$ not mentioned are not related to elements of $Q$ or the $a_{x}$'s, $b_{y}$'s unless they are forced to by transitivity. Notably, each $z \in C$ is not related to any $a_{x}$ or $b_{y}$.
    
    By universal-homogeneity, there is an embedding $R \hookrightarrow \mathbf{P}$ that fixes $\dur{p}$ pointwise. For ease of notation, we identify $a_{x}$'s, $b_{y}$'s, and $Q$ with their embedded copies in $\mathbf{P}$. Then $p_{Q} \in \mathcal{P}$, and is consistent with $p$ since $Q \cap \dur{p} = \emptyset$.
    
    By cofinality, let $q$ be a determined extension of $$p \cup p_{Q} \cup \brace{\paren{a_{x}, a_{x}} : x \in A} \cup \brace{\paren{b_{y}, b_{y}} : y \in B} \in \mathcal{P}.$$
    Then $f \in E \paren{p,q}$ implies there is $r \in \mathcal{P}$ such that $\dom{r} = \dur{q}$, $r$ fixes $\dur{p}$ pointwise, and $r \circ q \circ r^{-1} \subseteq f$. We abbreviate $a_{x}' := r \paren{a_{x}}$ for $x \in A$, $b_{y}' := r \paren{b_{y}}$ for $y \in B$, and $w' := r \paren{w}$ for $w \in Q$.
    
    We claim $Q' := r \brack{Q}$ and $p_{Q'} := r \circ p_{Q} \circ r^{-1}$ are as described in the theorem. First, note that $f \paren{a_{x}'} = \paren{r \circ q \circ r^{-1}} \paren{r \paren{a_{x}}} = r \paren{q \paren{a_{x}}} = a_{x}'$ for $x \in A$. Thus, $\orb{f}{a_{x}'} = \brace{a_{x}'}$; since also $x = r \paren{x} < r \paren{a_{x}}$, we have $\orb{f}{x} <_{f}^{s} \orb{f}{a_{x}'}$. Similarly, $f \paren{b_{y}'} = b_{y}'$, so $\orb{f}{b_{y}'} = \brace{b_{y}'} <_{f}^{s} \orb{f}{y}$ for $y \in B$. Moreover, for any $c \in Q$, since $a_{i} < w < b_{j}$, we have $a_{i}' < w' < b_{j}'$ and so $\orb{f}{x} <_{f}^{s} \orb{f}{a_{x}'} <_{f} \orb{f}{w'} <_{f} \orb{f}{b_{y}'} <_{f} \orb{f}{y}$. Also, for $c \in C$ and for $w \in Q$, $w \perp c$ implies $w' \perp c$ implies $\orb{f}{w'} \perp \orb{f}{c}$ by \Cref{OrbitalOrderCapturedByDeterm}.
\end{proof}

\begin{proof}[Proof of \Cref{GenericOrbitalQuotientIsRandom}]
    We show back-and-forth equivalence. ``Back'' is immediate from universal-homogeneity of $\mathbf{P}$, so we only prove ``forth''.
    
    Let $\gamma \from \mathbf{P} \rightharpoonup \OQ{f}{\mathbf{P}}$ be a finite partial isomorphism, and let $\ast \in \mathbf{P} \setminus \dom{\gamma}$. Let $\overline{A} \subseteq \dom{\gamma}$ be the set of immediate predecessors of $\ast$, let $\overline{B} \subseteq \dom{\gamma}$ be the set of immediate successors of $\ast$, and let $\overline{C} \subseteq \dom{\gamma}$ be the set of points incomparable to $\ast$. Let $T$ be a transversal of $\ran{\gamma}$, so that $\ran{\gamma} = \brace{\orb{f}{x} : x \in T}$. Also let $A,B,C \subseteq T$ such that $\gamma \brack{\overline{A}} = \brace{\orb{f}{x} : x \in A}$, $\gamma \brack{\overline{B}} = \brace{\orb{f}{y} : y \in B}$, and $\gamma \brack{\overline{C}} = \brace{\orb{f}{z} : z \in C}$. It follows that $A$, $B$, and $C$ satisfy the hypotheses for the orbital sandwich lemma (\Cref{Sandwich}). Letting $Q$ be any non-empty finite poset and $p_{Q}$ any partial automorphism, let $Q'$ and $p_{Q'}$ be as given by the lemma. Then for any $w \in Q'$ and all $x \in A$, $y \in B$, and $z \in C$ we have $\orb{f}{x} <_{f}^{s} \orb{f}{w} <_{f}^{s} \orb{f}{y}$, and $\orb{f}{w} \perp_{f}^{s} \orb{f}{z}$. By the sandwich principle, this implies $\gamma' := \gamma \cup \brace{\paren{\ast, \orb{f}{w}}}$ is a partial isomorphism. This proves ``forth'', and it follows that $\OQ{f}{\mathbf{P}} \isom \mathbf{P}$.

    We now show all types of orbitals are dense. Let $x,y \in \mathbf{P}$ such that $\orb{f}{x} <_{f}^{s} \orb{f}{y}$, and define the following finite posets and partial automorphisms:
    \begin{itemize}
        \item Let $Q_{\infty}$ have the following configuration:
        
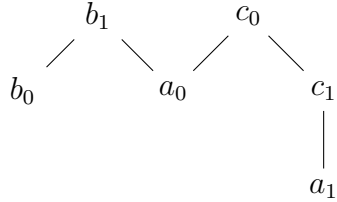
\begin{center}
            \begin{tikzpicture}
                \node (fy) at (-1,1) {$b_{1}$};
                \node (y) at (-2,0) {$b_{0}$};
                \node (x) at (0,0) {$a_{0}$};
                \node (z) at (1,1) {$c_{0}$};
                \node (fz) at (2,0) {$c_{1}$};
                \node (fx) at (2,-1.3) {$a_{1}$};
                \draw (y) -- (fy) -- (x) -- (z) -- (fz) -- (fx);
            \end{tikzpicture}
        \end{center}
        Then let $p_{Q_{\infty}} := \brace{\paren{a_{0}, a_{1}}, \paren{b_{0}, b_{1}}, \paren{c_{0}, c_{1}}}$;
        \item For $1 \leq n < \infty$, let $Q_{n} := \brace{d_{0}, \ldots, d_{n-1}}$ be an antichain of $n$ points, and $p_{Q_{n}} := \brace{\paren{d_{i}, d_{i+1}} : 0 \leq i < n}$ (where $d_{n} := d_{0}$).
    \end{itemize}
    By the orbital sandwich lemma (with $A = \brace{x}$, $B = \brace{y}$, and $C = \emptyset$), each of these finite partial automorphisms are realized by the action of $f$ on orbitals strictly between $\orb{f}{x}$ and $\orb{f}{y}$. In particular, each $p_{Q_{n}}$ gives an orbital of parity $0$ and spiral length $n$, and $p_{Q_{\infty}}$ gives an orbital of infinite spiral length, as well as orbitals of positive and negative parity (by considering $\orb{f}{a_{0}}$, $\orb{f}{b_{0}}$, and $\orb{f}{c_{0}}$ respectively).
\end{proof}

\subsection{Model-theoretic considerations}\label{subsec:ModelTheoreticConsiderations}

One may hope that other model-theoretic properties of $\mathbf{P}_{f}$ could have been used to classify the generic $f \in \Aut{\mathbf{P}}$. Unfortunately, beyond the universality and ultrahomogeneity displayed in \cref{subsec:CharacterizingGenericAutomorphisms}, the generic $\mathbf{P}_{f}$ fails to satisfy many of the common desirable model-theoretic properties. The main obstruction is essentially the following, which says orbitals of non-zero parity are rich enough to ``approximate'' points of infinite spiral length.

\begin{thm}\label{NonZeroParityHasAllSpiralLengths}
    For generic $f \in \Aut{\mathbf{P}}$, for all $x \in \mathbf{P}$ with $\parity{x,f} \neq 0$, and for all $1 \leq N < \infty$, there is $x' \sim_{f} x$ such that $\sp{x',f} = N$.
\end{thm}
\begin{proof}
    Assume without loss of generality $\parity{x,f} = +1$, since the case for negative parity is symmetric.
    
    We first show the case $N = 1$. Suppose $\sp{x,f} = n$. By genericity, let $p \in \mathcal{P}$ be determined such that $f \dhr_{\brace{f^{i} \paren{x} : 0 \leq i < n}} \subseteq p \subseteq f$. Furthermore, let $k$ be large enough so that $\mathbf{b}_{\geq k} \paren{y,x}$ and $\mathbf{b}_{\leq -k} \paren{x,y}$ are constant for all $y \in \dur{p}$. Replace $p$ with $p \cup f \dhr_{\brace{f^{i} \paren{x} : 0 \leq i < k+n}}$; this is an economical extension of $p$, so it remains determined. (We abuse notation and let $p$ be this extension instead.) Then we define the poset on $\dur{p} \sqcup \brace{a_{0}, a_{1}}$ by letting $a_{i}$ and $f^{k+i} \paren{x}$ have the same order-types over $\dur{p} \setminus f^{\Z} \paren{x}$ for $i=0,1$, and letting $f^{k} \paren{x} < a_{0} < f^{k+n} \paren{x}$, $f^{k+1} \paren{x} < a_{1}$, and $a_{0} < a_{1}$ as in \Cref{fig:HasLength1Spiral}.
    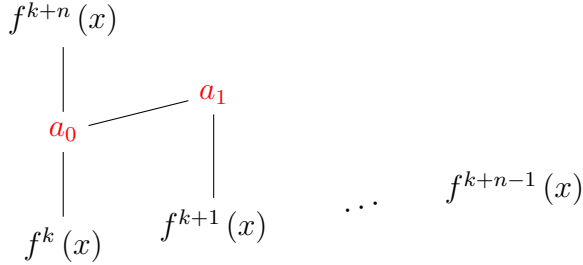
\begin{figure}[h!]
        \centering
        \begin{tikzpicture}
            \node (x) at (0,-1.5) {$f^{k} \paren{x}$};
            \node (fx) at (2,-1.25) {$f^{k+1} \paren{x}$};
            \node (f2x) at (4,-1) {\rotatebox[origin=c]{15}{$\cdots$}};
            \node (f3x) at (6,-0.75) {$f^{k+n-1} \paren{x}$};
            \node (f4x) at (0,1.5) {$f^{k+n} \paren{x}$};
            \node (a0) at (0,0) {$\color{red}a_{0}$};
            \node (a1) at (2,0.5) {$\color{red}a_{1}$};
            \draw (x) -- (a0) -- (f4x);
            \draw (a1) -- (fx);
            \draw (a0) -- (a1);
        \end{tikzpicture}
        \caption{Forcing $\orb{f}{x}$ to contain a point with spiral length $1$.}
        \label{fig:HasLength1Spiral}
    \end{figure}
    
    By choice of $k$, $f^{k'} \paren{x}$ and $f^{k} \paren{x}$ have the same order-types over $\dur{p} \setminus f^{\Z} \paren{x}$ for all $k' \geq k$; thus, the addition of the $a_{i}$'s does not create new relations between points of $\dur{p}$. Moreover, the map $p \cup \brace{\paren{a_{0}, a_{1}}}$ is a partial automorphism.
    
    Then by universal-homogeneity, the $a_{i}$'s may be realized as points in $\mathbf{P}$, and $p \cup \brace{\paren{a_{0}, a_{1}}} \in \mathcal{P}$. Let $p' \in \mathbf{P}$ be a determined extension of $p \cup \brace{\paren{a_{0}, a_{1}}}$; by $f \in E \paren{p,p'}$, take $r \in \mathcal{P}$ such that $\dom{r} = \dur{p'}$, $r$ fixes $\dur{p}$ pointwise, and $r \circ p' \circ r^{-1} \subseteq f$. Then we claim $x' := r \paren{a_{0}}$ is the desired point. Indeed, we have $f^{k} \paren{x} < r \paren{a_{0}} < f^{k+n} \paren{x}$ and so $r \paren{a_{0}} \sim_{f} x$, and $f \paren{r \paren{a_{0}}} = r \paren{p' \paren{a_{0}}} = r \paren{a_{1}} > r \paren{a_{0}}$ implies $\sp{r \paren{a_{0}}, f} = 1$.
    
    Now suppose $\sp{x,f} = 1$ (so $f^{\Z} \paren{x}$ is linearly ordered); we show how to obtain $x' \sim_{f} x$ with any given spiral length $N$. As above, choose $k$ large enough so that $\mathbf{b}_{\geq k} \paren{y,x}$ and $\mathbf{b}_{\leq -k} \paren{x,y}$ are constant for all $y \in \dur{p}$, and replace $p$ with $p \cup f \dhr_{\brace{f^{i} \paren{x} : 0 \leq i < k+N}}$. Then define the poset $\dur{p} \sqcup \brace{a_{i} : 0 \leq i \leq N}$ by letting $a_{i}$ and $f^{k+i} \paren{x}$ have the same type over $\dur{p} \setminus f^{\Z} \paren{x}$, and letting $f^{k+i} \paren{x} < a_{i}$, $a_{0} < f^{k+N} \paren{x}$, and $a_{i} \perp a_{j}$ for $0 < \absval{i-j} < N$ as in \Cref{fig:HasLengthNSpiral}.
    \begin{figure}[h!]
        \centering
        \begin{tikzpicture}
            \node (x) at (0,0) {$f^{k} \paren{x}$};
            \node (fx) at (0,1.25) {$f^{k+1} \paren{x}$};
            \node (f2x) at (0,2.5) {$f^{k+2} \paren{x}$};
            \node (f3x) at (0,3.75) {$\vdots$};
            \node (f4x) at (0,5) {$f^{k+N} \paren{x}$};
            \node (a0) at (-2.5,1.25) {$\color{red}a_{0}$};
            \node (a1) at (-2.5,2.5) {$\color{red}a_{1}$};
            \node (a2) at (-2.5,3.75) {$\color{red}a_{2}$};
            \node (a3) at (-2.5,5.25) {$\vdots$};
            \node (a4) at (-2.5,6.25) {$\color{red}a_{N}$};
            \draw (x) -- (fx) -- (f2x) -- (f3x) -- (f4x);
            \draw (x) -- (a0) -- (f4x) -- (a4);
            \draw (fx) -- (a1);
            \draw (f2x) -- (a2);
        \end{tikzpicture}
        \caption{Forcing $\orb{f}{x}$ to contain a point with spiral length $N$.}
        \label{fig:HasLengthNSpiral}
    \end{figure}
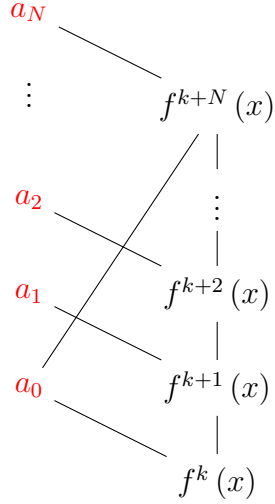
    
    The rest of the argument proceeds similarly to the $N = 1$ case: by choice of $k$, we may realize the $a_{i}$'s as elements of $\mathbf{P}$, and $p \cup \brace{\paren{a_{i},a_{i+1}} : 0 \leq i < N} \in \mathcal{P}$. Take $p'$ a determined extension of this partial automorphism, and by $f \in E \paren{p,p'}$, take $r \in \mathcal{P}$ accordingly. Then $x' := r \paren{a_{0}}$ works, since $f^{k} \paren{x} < r \paren{a_{0}} < f^{k+N} \paren{x}$ implies $r \paren{a_{0}} \sim_{f} x$, and moreover $\sp{r \paren{a_{0}}, f} = N$.
\end{proof}

\begin{lem}\label{LDefinableRelationsInSigma}
    For generic $f \in \Aut{\mathbf{P}}$, the following are $0$-definable relations in $\mathbf{P}_{f}$:
    \begin{enumerate}
        \item $\sp{x,f} = \infty$;
        \item $\orb{f}{x} \leq_{f}^{w} \orb{f}{y}$;
        \item $\orb{f}{x} <_{f}^{s} \orb{f}{y}$;
        \item $x \sim_{f} y$.
    \end{enumerate}
\end{lem}
\begin{proof}
    (a), (b), and (c) are immediate from \Cref{PropertiesOfSigma}. For (d), note $x \sim_{f} y$ iff both $\orb{f}{x} \leq_{f}^{w} \orb{f}{y}$ and $\orb{f}{y} \leq_{f}^{w} \orb{f}{x}$.
\end{proof}

\begin{thm}[Negative results]
    Let $f \in \Aut{\mathbf{P}}$ be generic.
    \begin{enumerate}
        \item $\operatorname{Th}_{L} \paren{\mathbf{P}_{f}}$ (that is, the $L$-theory of $\mathbf{P}_{f}$) is not $\omega$-categorical.
        \item $\mathbf{P}_{f}$ is not $\omega$-saturated.
        \item The relation $\sp{x,f} = \infty$ is not quantifier-free definable. Thus, $\operatorname{Th}_{L} \paren{\mathbf{P}_{f}}$ does not admit quantifier elmination.
    \end{enumerate}
\end{thm}
\begin{proof}
    \begin{enumerate}
        \item By \Cref{NonZeroParityHasAllSpiralLengths}, let $x_{n} \in \mathbf{P}$ with $\sp{x_{n},f} = n$ (which is definable by \Cref{BDefinableRelations}). Then the $\tp_{L} \paren{x_{n}, \emptyset}$'s are an infinite collection of distinct $1$-types, so $\operatorname{Th}_{L} \paren{\mathbf{P}_{f}}$ is not $\omega$-categorical by Ryll-Nardzewski's theorem.

        \item Let $p \paren{x}$ be the set of $L$-formulas given by: $$p \paren{x} = \brace{\neg b_{i} \paren{x,x} : i \neq 0} \cup \brace{\neg \exists y \exists z \, \mu \paren{x,y,z}},$$
        where $\mu \paren{x,y,z}$ is a formula saying that $x$, $y$, and $z$ form an ``M'' configuration as in \Cref{fig:MExample}. Then for any finite set $p_{0} \paren{x} \subseteq p \paren{x}$, we have that $p_{0} \paren{x}$ holds whenever $$\max \brace{\absval{i} : i \in \Z, \neg b_{i} \paren{x,x} \in p_{0} \paren{x}} < \sp{x,f} < \infty.$$
        Moreover, such points exist in $\mathbf{P}_{f}$ by \Cref{NonZeroParityHasAllSpiralLengths}. Thus, $p \paren{x}$ is a partial type over $\emptyset$. However, $p \paren{x}$ cannot be realized in $\mathbf{P}_{f}$, since a realization would be a point with infinite spiral length but without a witnessing ``M'' configuration, contradicting \Cref{PropertiesOfSigma}\ref{CanFitMs}.

        \item Let $\phi \paren{x}$ be a quantifier-free $L$-formula in one variable. We claim $\phi$ cannot define the relation $\sp{x,f} = \infty$, so assume $\mathbf{P}_{f} \models \phi \paren{x}$ for all $x \in \mathbf{P}$ with $\sp{x,f} = \infty$, and we will show that $\mathbf{P}_{f} \models \phi \paren{y}$ for some $y \in \mathbf{P}$ with $\sp{y,f} < \infty$.
    
        Each atomic subformula of $\phi \paren{x}$ is of the form $b_{i} \paren{x,x}$ for some $i \in \Z$, which holds if $i = 0$ and fails otherwise (since $\sp{x,f} = \infty$). Using \Cref{NonZeroParityHasAllSpiralLengths}, let $y \in \mathbf{P}$ such that $\absval{i} < \sp{y,f} < \infty$ for all $i \in \Z$ such that $b_{i} \paren{x,x}$ appears in $\phi \paren{x}$. Then $b_{i}^{f} \paren{x,x} \iff b_{i}^{f} \paren{y,y}$ for all $\absval{i} < \sp{y,f}$, and so $\mathbf{P}_{f} \models \phi \paren{y}$. \qedhere
    \end{enumerate}
\end{proof}

We remark that the failure of these properties for $\mathbf{P}_{f}$ contrasts with $\mathbf{P}$ (as a poset), which has these properties, as do many other \Fraisse{} limits in finite languages.

\section{Closing remarks}
A number of questions remain open.

\begin{qns}
    \begin{itemize}
        \item We have shown that the generic orbital quotient of $\mathbf{P}$ is isomorphic to $\mathbf{P}$ itself, and all reasonable combinations of parity and spiral length are dense in the orbital quotient. However, unlike Truss's result for $\Q$, we are unable to show that this itself characterizes the generic automorphism, even with the other properties shown to hold for generics. \textit{Can the conclusion of \Cref{GenericOrbitalQuotientIsRandom} replace any of the properties defining $\Sigma$, or the conclusions of \Cref{PropertiesOfSigma}, to obtain another characterization of the generic automorphism of $\mathbf{P}$?}
        \item Note that \Cref{GenericOrbitalQuotientIsRandom} considers $\OQ{f}{\mathbf{P}}$ as a partial order with respect to the strong order $<_{f}^{s}$. The same result \emph{as written} does not hold for $<_{f}^{w}$. Indeed, the set of parity-$0$ orbitals with spiral length $1$ (that is, fixed points) cannot be dense in this order, since for any $\orb{f}{x} <_{f}^{w} \orb{f}{y}$ with $\orb{f}{x} \nless_{f}^{s} \orb{f}{y}$, a fixed point between them would contradict \Cref{FixedPointBetweenForcesStronglyLess}. \textit{Does some natural analogue of \Cref{GenericOrbitalQuotientIsRandom} hold when $\OQ{f}{\mathbf{P}}$ is given the weak order $<_{f}^{w}$? If so, can this be used in an explicit characterization of the generic automorphism of $\mathbf{P}$?}
        \item By \Cref{PropertiesOfSigma}, $\mathbf{P}_{f}$ exhibits enough universal-homogeneity to force certain configurations to occur. Indeed, we have not ruled out that the other properties defining $\Sigma$ can be similarly derived. \textit{Is the given list of properties characterizing the generic automorphism of $\mathbf{P}$ minimal?}
    \end{itemize}
\end{qns}

Finally, we discuss applications to other structures. Suppose $\mathbf{K}$ is the limit of another \Fraisse{} class, such that $\Aut{\mathbf{K}}$ admits generic elements. Can we associate to each $f \in \Aut{\mathbf{K}}$ an expanded structure $\mathbf{K}_{f}$ whose description can characterize the generic automorphisms? Such a structure must be inter-definable with $\paren{\mathbf{K},f}$ --- the result of adjoining a single function symbol to $\mathbf{K}$ --- so that the isomorphism type of $\mathbf{K}_{f}$ reflects the conjugacy relation in $\Aut{\mathbf{K}}$. In the case of $\mathbf{P}$, this used antisymmetry of partial orders in a seemingly unavoidable way. We also may expect that such a structure may exhibit some universal-homogeneity property, in the same way our $\mathbf{P}_{f}$ does.

\bibliographystyle{alpha}
\bibliography{refs}

\end{document}